\documentclass[11pt, a4paper]{amsart}
\usepackage{amsfonts,amsmath, amsthm, amssymb, amscd}
\usepackage{latexsym}
\input{amssym.def}
\input{amssym}

\usepackage[dvips]{graphicx}
\usepackage{psfrag}

\newtheorem{theorem}{Theorem}[section]
\newtheorem{lemma}[theorem]{Lemma}
\newtheorem{cor}[theorem]{Corollary}
\newtheorem{prop}[theorem]{Proposition}
\newtheorem{dfn}[theorem]{{Definition}}
	
\newtheorem*{rmk}{{Remark}}
\newtheorem*{exbdeig}{{Examples of bounded eigenfunctions}}
\newtheorem*{thm1}{{Theorem 1.4}}
\newtheorem*{thm2}{{Theorem 1.5}}
\newtheorem*{thm3}{{Theorem 1.6}}
\newtheorem*{thm4}{{Theorem 1.7}}

\numberwithin{equation}{section}
\newcommand {\N}{\mathbb{N}} %% positive integers
\newcommand {\R}{\mathbb{R}} %% reals
\newcommand {\HH}{{\mathcal H}} %% horosphere

\DeclareMathOperator{\id}{id}
\DeclareMathOperator{\vol}{vol}

\DeclareMathOperator{\supp}{supp}

\DeclareMathOperator{\grad}{grad}
\DeclareMathOperator{\rank}{rank}
\DeclareMathOperator{\tr}{tr}

\DeclareMathOperator{\Jac}{Jac}

\begin{document}
%%%%%%%%%
% eingereicht:16.04.2014;Geometric Analysis
%%%%%
\title[Rank one asymptotically harmonic manifolds]{Harmonic functions on rank one asymptotically harmonic manifolds}

\author{Gerhard Knieper and Norbert Peyerimhoff }
\date{\today}
\address{Faculty of Mathematics,
Ruhr University Bochum, 44780 Bochum, Germany}
\email{gerhard.knieper@rub.de}
\address{Department of Mathematical Sciences, Durham University, Durham DH1 3LE, UK}
\email{norbert.peyerimhoff@dur.ac.uk}

%%%%%%%%%%%%%%%%%%%%%%%%%%%%%%%%%%%%%%%%%%%%%%%%%%%%%%%%%%%%%%%%%%%%%%

\begin{abstract}
  Asymptotically harmonic manifolds are simply connected complete
  Riemannian manifolds without conjugate points such that all
  horospheres have the same constant mean curvature $h$. In this
  article we present results for harmonic functions on rank one
  asymptotically harmonic manifolds $X$ with mild curvature
  boundedness conditions. Our main results are (a) the explicit
  calculation of the Radon-Nykodym derivative of the visibility
  measures, (b) an explicit integral representation for the solution
  of the Dirichlet problem at infinity in terms of these visibility
  measures, and (c) a result on horospherical means of bounded
  eigenfunctions implying that these eigenfunctions do not admit
  non-trivial continuous extensions to the geometric
  compactification $\overline{X}$.
\end{abstract}

%%%%%%%%%%%%%%%%%%%%%%%%%%%%%%%%%%%%%%%%%%%%%%%%%%%%%%%%%%%%%%%%%%%%%%

\maketitle

\tableofcontents
%%%%%%%%%%%%%%%%%%%%%%%%%%%%%%%%%%%%%%%%%%%%%%%%%%%%%%%%%%%%%%%%%%%%%%

\section{Introduction}

Manifolds with asymptotically harmonic metrics were first introduced
by Ledrappier (\cite[Thm. 1]{Led}) in the special case of negative
curvature in connection with rigidity of measures related to the
Dirichlet problem (harmonic measure) and the dynamics of the geodesic
flow (Bowen-Margulis measure). One of the equivalent characterisations
of asymptotically harmonic metrics there was that all horospheres have
constant mean curvature $h \ge 0$. We express this geometric property
in terms of Jacobi tensors (see Definion \ref{def:asymharm}
below). Let $(X,g)$ be a complete Riemannian manifold without
conjugate points and let $\pi: SX \to X$ be the canonical footpoint
projection from the unit tangent bundle. For $v \in SX$ let $c_v: {\mathbb
  R} \to X$ be the unique geodesic given by $c_v'(0) = v$. Let
$S_{v,r}$ and $U_{v,r}$ be the orthogonal Jacobi tensors along $c_v$,
defined by $S_{v,r}(0) = U_{v,r}(0) = {\rm id}$ and $S_{v,r}(r) = 0$
and $U_{v,r}(-r) = 0$. Note that we have $U_{v,r}(t) =
S_{-v,r}(-t)$. The stable and unstable Jacobi tensors $S_v$ and $U_v$
are then defined as the Jacobi tensors along $c_v$ with initial
conditions $S_v(0) = U_v(0) = {\rm id}$ and $S_v'(0) = \lim_{r \to
  \infty} S_{v,r}'(0)$ and $U_v'(0) = \lim_{r \to \infty}
U_{v,r}'(0)$. They are related by $U_v(t) = S_{-v}(-t)$. For
simplicity of notation, we introduce $U(v) = U_v'(0)$ and $S(v) =
S_v'(0)$. (For more detailed information on Jacobi tensors see, e.g.,
\cite{Kn1}.)

\begin{dfn} \label{def:asymharm}
  An {\em asymptotically harmonic manifold} $(X,g)$ is a complete,
  simply connected Riemannian manifold without conjugate
  points such that for all $v \in SX$ we have $\tr U(v) = h$ for a
  constant $h \ge 0$.
\end{dfn}

The manifolds considered in this article are {\em rank one}
asymptotically harmonic manifolds. The notion of rank has been
introduced by Ballmann, Brin and Eberlein in \cite{BBE} for
nonpositively curved manifolds as the dimension of the parallel Jacobi
fields along geodesics. Since we do not assume nonpositive curvature,
the notion of rank has to be understood in the following generalized
sense given in \cite[Def. 3.1]{Kn2}:

\begin{dfn} Let $(X,g)$ be a complete simply connected Riemannian
  manifold without conjugate points. For $v \in SX$ let $D(v) =
  U(v)-S(v)$ and we define
  $$
  \rank(v) = \dim (\ker D(v)) +1
  $$
  and
  $$
  \rank(X) = \min  \{ \rank(v) \mid  v \in SX \}.
  $$
\end{dfn}

In \cite{KnPe2}, we proved equivalence of the following four
properties for asymptotically harmonic manifolds $X$ under the mild
curvature boundedness condition
\begin{equation} \label{eq:curvboundcond} \Vert R \Vert \le R_0 \quad
  {\rm and} \quad \Vert \nabla R \Vert \le R_0'
\end{equation}
for some constants $R_0,R_0' > 0$: (a) $X$ has rank one, (b) $X$ has
Anosov geodesic flow, (c) $X$ is Gromov hyperbolic, and (d) $X$ has
purely exponential volume growth with growth rate $h_{vol}=h$. These
equivalences were first proved for noncompact harmonic manifolds in
\cite{Kn2} and for asymptotically harmonic manifolds admitting compact
quotients in \cite{Zi1}. Besides negatively curved symmetric spaces,
Damek-Ricci spaces provide examples of rank one harmonic and therefore
also asymptotically harmonic manifolds, since they all have purely
exponential volume growth. As a consequence, all Damek-Ricci spaces are
Gromov hyperbolic. (Note that all non-symmetric Damek-Ricci
spaces admit zero-curvature.) In this article, we use the above 
equivalences to study harmonic functions on rank one asymptotically
harmonic manifolds $(X,g)$ satisfying \eqref{eq:curvboundcond}. Let us
discuss the results of this paper in more detail.

In Sections 2 and 3, we introduce the geometric boundary $X(\infty)$
via equivalence classes of geodesic rays and the canonical maps
$\varphi_p: S_pX \to X(\infty)$, $\varphi_p(v) = c_v(\infty)$. These
maps have natural extensions $\bar \varphi_p$ to the geometric
compactification $\overline{X} = X \cup X(\infty)$, and we show that
these extensions are homeomorphisms. The visibility measures $\{ \mu_p
\}$ on $X(\infty)$ are then defined as follows:

\begin{dfn} \label{def:vismeas} Let ${\mathcal M}_1(X(\infty))$ denote
  the space of Borel probability measures on $X(\infty)$. For every $p
  \in X$, let $\mu_p \in {\mathcal M}_1(X(\infty))$ be defined by
  $$ \int_{X(\infty)} f(\xi)\, d\mu_p(\xi) = \frac{1}{\omega_n}
  \int_{S_pX} f(\varphi_p(v))\, d\theta_p(v) \quad \forall\, f \in
  C(X(\infty)), $$ where $n= {\rm dim}(X)$ and $\omega_n$ is the
  volume of the $(n-1)$-dimensional standard unit sphere and
  $d\theta_p$ is the volume element of $S_pX$ induced by the
  Riemannian metric. $\mu_p$ is called the {\em visibility measure} of
  $(X,g)$ at the point $p$.
\end{dfn}

Sections 4 and 5 are concerned with the explicit calculation of the
Radon-Nykodym derivative of the visibility measures. To state the
result (Theorem \ref{thm:radon-nykodym} below), we need Busemann
functions. Let $v \in S_qX$ and $\xi = c_v(\infty) \in
X(\infty)$. Then the Busemann function (associated to $v \in S_qX$ or
to $(q,\xi) \in X \times X(\infty)$) is defined as
\begin{equation} \label{eq:Busemann}
b_v(p) = b_{q,\xi}(p) = \lim_{t \to \infty} d(c_v(t),p)-t. 
\end{equation}

\begin{theorem} \label{thm:radon-nykodym} Let $(X,g)$ be a rank one
  asymptotically harmonic manifold satisfying
  \eqref{eq:curvboundcond}. Let $(\mu_p)_{p \in X}$ be the associated
  family of visibility measures. Then these measures are pairwise
  absolutely continuous and we have
  $$ \frac{d\mu_p}{d\mu_q}(\xi) = e^{-h b_{q,\xi}(p)}. $$
\end{theorem}

An analogous result on the Radon-Nykodym derivative for asymptotically
harmonic manifolds in the case of pinched negative curvature was given
in \cite[Prop. 6.1]{CaSam}.

Since our rank one asymptotically harmonic manifolds $(X,g)$ are
Gromov hyperbolic and have positive Cheeger constants (see
\cite[Prop. 5.3]{KnPe2}), the general theory of Ancona
\cite{Anc1,Anc2} implies that the geometric boundary and the Martin
boundary agree and that the Dirichlet problem at infinity can be
solved. In Section 6 we give an alternative direct proof of this
latter fact and give an explicit integral representation for the
solution of the Dirichlet problem at infinity in terms of the
visibility measures:

\begin{theorem} \label{thm:dirichprob} Let $(X,g)$ be a rank one
  asymptotically harmonic manifold satisfying
  \eqref{eq:curvboundcond}. Let $f: X(\infty) \to {\R}$ be a
  continuous function. Then there exists a unique harmonic function
  $H_f: X \to {\R}$ such that
  \begin{equation} \label{eq:convdirich}
  \lim\limits_{x \to \xi} H_f(x) = f(\xi).
  \end{equation}
  Moreover, $H_f$ has the following integral presentation:
  $$
  H_f(x) = \int\limits_{X(\infty)} f(\xi) d \mu_x(\xi),
  $$
  where $\{\mu_x\}_{x \in X} \subset {\mathcal M}_1(X(\infty))$ are
  the visibility probability measures.
\end{theorem}

A related result in the setting of harmonic manifolds can be found in
Zimmer \cite[Thm. 1]{Zi2}. Moreover, the solution of the Dirichlet problem at
infinity for general nonpositively curved rank one manifolds admitting
compact quotients was shown by Ballmann \cite{Ba}.
 
In Section 8 we consider eigenfunctions $\Delta f + \lambda f = 0$,
$\lambda \in {\mathbb R} \backslash \{0\}$ on rank one asymptotically
harmonic manifolds $X$ satisfying \eqref{eq:curvboundcond}. We show
that if such an eigenfunction $f \in C^\infty(X,{\mathbb C})$ has a
continuous extension to the boundary $X(\infty)$, then the extension
must be necessarily trivial, in contrast to Theorem
\ref{thm:dirichprob} for harmonic functions. The proof is based on
taking horospherical means. Since horospheres $\mathcal H$ are
noncompact, the averages have to be taken via {\em compact
  exhaustions} $\{ K_j \}$ with smooth boundaries $\partial K_j$. We
first observe in Section 8 (see Theorem \ref{thm:meanvalprop0})
that, for continuous functions $f: \overline{X} = X \cup X(\infty) \to
{\mathbb R}$ and horospheres $\mathcal{H}$ centered at $\xi \in
X(\infty)$ with compact exhaustion $\{ K_j \}$, we have
\begin{equation} \label{eq:horosphmean} \lim\limits_{j \to \infty}
  \frac{\int_{K_j} f(x) dx}{\vol_{n-1}(K_j)} = f(\xi).
\end{equation}
The expression \eqref{eq:horosphmean} is called {\em the horospherical
  mean of $f$} with respect to the exhaustion $\{ K_j \}$. In Section
7, we prove that all horospheres in these spaces have polynomial
volume growth, which implies that they admit {\em (compact)
  isoperimetric exhaustions} $\{ K_j \}$, that is,
\begin{equation} \label{eq:isopexh}
  \frac{\vol_{n-2}(\partial K_j)}{\vol_{n-1}(K_j)} \to 0 \quad
  \text{as}\; j \to \infty.
\end{equation}
The main result of Section 8 is that, for all $\lambda
\in {\mathbb R}\backslash \{0\}$, the horospherical means (with
respect to isoperimetric exhaustions) of bounded eigenfunctions are
zero.

\begin{theorem} \label{thm:horomeanbdeigfunc} Let $(X,g)$ be a rank
  one asymptotically harmonic manifold of dimension $n$ satisfying
  \eqref{eq:curvboundcond} and $h > 0$ be the mean curvature of all
  horospheres. Let $\lambda \neq 0$ be a real number and $f \in
  C^\infty(X)$ be a bounded function satisfying $\Delta f
  + \lambda f = 0$ and $\mathcal{H} \subset X$ be a horosphere with
  {\em isoperimetric exhaustion} $\{ K_j \}$. Then we have
  \begin{equation} \label{eq:horosphmean2} \lim\limits_{j \to \infty}
    \frac{\int_{K_j} f(x) dx}{\vol_{n-1}(K_j)} = 0.
  \end{equation}
\end{theorem}

This result leads to the following above mentioned fact,
complementing Theorem \ref{thm:dirichprob}.

\begin{theorem} \label{thm:dirichprob-eigen} Let $(X,g)$ be a rank one
  asymptotically harmonic manifold satisfying
  \eqref{eq:curvboundcond}. Let $\lambda \in {\mathbb R} \backslash
  \{0\}$ and $f \in C^\infty(X)$ be an eigenfunction $\Delta f +
  \lambda f = 0$. If $f$ has a continuous extension $F \in
  C(\overline{X})$ then we have necessarily $F \vert_{X(\infty)}
  \equiv 0$.
\end{theorem}

{\bf Acknowledgement:} The authors would like to thank Evangelia Samiou for
bringing the references  \cite{ItSa1, ItSa2} to their attention.

\section{Uniform divergence of geodesics}
\label{chap:unifdifgeod}

In this section, we prove that for every distance $d > 0$ and any
angle $\alpha > 0$ there exists a $t_0 > 0$, such that any two unit
speed geodesics $c_1,c_2$ starting at the same point and differing by
an angle $\ge \alpha$ will diverge uniformly in the sense that
$d(c_1(t),c_2(t)) \ge d$ for all $t \ge t_0$. For the proof, we start
with the following lemma.

\begin{lemma} \label{lem:jacrep} Let $(X,g)$ be a manifold without
  conjugate points and, for $v \in SX$, let $A_v$ be the orthogonal
  Jacobi tensor along $c_v$ satisying $A_v(0) = 0$ and $A_v'(0) = {\rm
    id}$. Then we have
  \begin{itemize}
  \item[(i)] $A_v(t) = U_v(t) \int\limits_0^t (U_v^* U_v)^{-1}(u)
    du$,
  \item[(ii)] $( U_v'(0) -S_{v,t}'(0))^{-1} = \int\limits_0^t
    (U_v^* U_v)^{-1}(u) du$.
  \end{itemize}
\end{lemma}

\begin{proof}
  Since the endomorphism $U_v(u)$ is non-singular and Lagrangian for
  all $u \in {\mathbb R}$, we conclude from \cite[Prop. 2.1]{Kn2} that
  $$ A_v(t) = U_v(t)\left( \int_0^t (U_v^*U_v)^{-1}(u)du \; C_1 + C_2 \right) $$
  with suitable constant tensors $C_1$ and $C_2$. Evaluating and
  differentiating this identity at $t=0$ yields $C_2 = 0$ and $C_1 =
  {\rm id}$, finishing the proof of (i). The statement (ii) can be found
  in \cite[Lemma 2.3]{Kn2}.
\end{proof}

\begin{prop} \label{prop:Avest}
  Let $(X,g)$ be a rank one asymptotically harmonic manifold
  satisfying \eqref{eq:curvboundcond}. Then there exist constants $a,
  \rho > 0$ such that
  $$
  \|A_v(t)\, x\| \ge a e^{\frac{\rho}{2}t} \| x \|
  $$
  for all $v \in SX$, $x \in v^\bot \subset TX$ and $t \ge 1$. 
\end{prop}

\begin{proof}
  We conclude from \cite[Thm. 1.3]{KnPe2} that there exists $\rho > 0$
  such that $D(v) = U(v) - S(v) \ge \rho \cdot {\rm id}$. Using
  \cite[Prop. 2.5]{KnPe2} and the fact that $S_v(t)$ is non-singular
  for $t \ge 0$, we conclude that there exists $a_2 > 0$ such that
  $\Vert S_v^{-1}(t) y \Vert \ge \frac{1}{a_2} e^{\frac{\rho}{2}t} \Vert y
  \Vert$ for all $y \in (\Phi^tv)^\bot$, where $\Phi^t: SX \to SX$
  denotes the geodesic flow. Using
  $S_{\Phi^t w}(u) y_u = S_w(u+t)(S_w^{-1}(t)y)_u$ (where $y_u$ is the
  parallel translation of $y \in (\Phi^t w)^\bot$ along $c_v$) with $u=-t$
  and $-v=\Phi^tw$ yields
  $$ \Vert U_v(t) y \Vert =  \Vert S_{-v}(-t)y \Vert = 
  \Vert S_{-\Phi^t v}^{-1}(t) y_t \Vert \ge 
  \frac{1}{a_2} e^{\frac{\rho}{2}t} \Vert y \Vert. $$ 
  Lemma \ref{lem:jacrep} yields
  $$ \Vert A_v(t) (U_v'(0) - S_{v,t}'(0)) y \Vert = 
  \Vert U_v(t) y \Vert \ge \frac{1}{a_2} e^{\frac{\rho}{2}t} \Vert y \Vert, $$ 
  i.e.,
  $$ \Vert A_v(t) x \Vert \ge \frac{1}{a_2 \Vert U_v'(0) - S_{v,t}'(0) \Vert} 
  e^{\frac{\rho}{2}t} \Vert x \Vert \ge \frac{1}{a_2(\Vert U_v'(0) \Vert + \Vert
    S_{v,t}'(0) \Vert)} e^{\frac{\rho}{2}t} \Vert x \Vert. $$
  The proposition follows now from $\Vert U_v'(0) \Vert \le \sqrt{R_0}$ and
  $$ \Vert S_{v,t}'(0) \Vert = \Vert A_v'(t)A_v^{-1}(t) \Vert \le \sqrt{R_0}
  \coth(\sqrt{R_0}), $$
  which can be found in \cite[Lem. 2.2]{KnPe2}
\end{proof}
  
Using this we derive the uniform divergence of geodesics described above.

\begin{cor}\label{cor:uniformdiv}
  Let $c_v : [0, \infty) \to X$ and $c_w : [0, \infty) \to X$ be two
  geodesics with $v,w \in S_pX$. Then
  $$
  d(c_v(t), c_w(t)) \ge a(t)  \angle(v,w)
  $$
  where $a: [0,\infty) \to [0, \infty)$ is a function (not depending
  on $p \in X$) with $\lim\limits_{t \to \infty} a(t) = \infty$.
\end{cor}

\begin{proof}
 Let $c: [0,1] \to X$ be a geodesic connecting $c_v(t)$ with
  $c_w(t)$. Then $c$ is given by
  $$
  c(s) = \exp_p r(s) v(s),
  $$
  where $v(s) \in S_p X$ and $r(s) > 0$ for all $0 \leq s \leq 1$, and
  $v(0) = v$, $v(1) =w$ and $r(0) = r(1) = t$. Then
  \begin{eqnarray*}
  c'(s_0) &=& D \exp_p (
  r(s_0) v(s_0) )(r'(s_0) v(s_0) +  r(s_0) v'(s_0))\\
  &=& r'(s_0) c_{v(s_0)}' (r(s_0)) + A_{v(s_0)} ( r(s_0)) (v'(s_0)).
  \end{eqnarray*}
  Since $c_{v(s_0)}' (r(s_0)) \perp A_{v(s_0)} (r(s_0))
  (v'(s_0))$, we obtain
  \begin{eqnarray*}
    \left \| c'(s_0) \right \|^2 &= &
    (r'(s_0))^2 + ||A_{v(s_0)} ( r(s_0)) v'(s_0)||^2\\
    & \geq& ||A_{v(s_0)} ( r(s_0)) v'(s_0)||^2.
  \end{eqnarray*}
  If there exists $s_0 \in [0, 1]$ such that $r(s_0) \le \frac{t}{2}$
  then using the triangle inequality we have $d(c_v(t), c_w(t)) \ge
  t \ge (t/\pi) \angle(v,w)$. If this is not the case, we obtain for all $t > 0$
  \begin{eqnarray*}
    d(c_v(t), c_w(t)) = {\rm length}(c) &\ge& \int\limits_0^1 ||A_{v(s)} ( r(s)) v'(s)|| ds \\
    &\ge& a e^{\frac{\rho}{4}t} \int\limits_0^1||v'(s)|| ds\\
    &\ge& a e^{\frac{\rho}{4}t} \angle(v,w).
  \end{eqnarray*}
  The corollary follows now with the choice
  $$ a(t) = \min \left\{ \frac{t}{\pi}, a e^{\frac{\rho}{4}t} \right\}. $$ 
\end{proof}

\begin{rmk}
  Note that the proof shows that the function $a(t)$ describing the
  divergence of geodesics has at least linear growth.
\end{rmk}

\section{The geometric compactification}
\label{sec:geomcomp}

Let $(X,g)$ be a rank one asymptotically harmonic manifold satisfying
\eqref{eq:curvboundcond}. The {\em geometric boundary} $X(\infty)$ is
the set of equivalence classes of asymptotic geodesic rays. Two
geodesic rays $c_1, c_2: [0,\infty) \to X$ are called {\em
  asymptotic}, if there exists $C > 0$ with $d(c_1(t),c_2(t)) \le C$
for all $t \ge 0$. The equivalence class of a geodesic ray $c$ is
denoted by $c(\infty)$.

Let $p \in X$ and consider the map $\varphi_p: S_pX \to X(\infty)$
with $\varphi_p(v) = c_v(\infty)$. Uniform divergence of geodesics
implies that $\varphi_p$ is injective. Our next aim is to prove
surjectivity of $\varphi_p$. For this, we first prove general results
which will also be useful later on. The first result requires besides no conjugate points only a lower bound
on the sectional curvature of $X$.

\begin{prop} \label{prop:inexdist} Let $p,q_0,p_0 \in X$ be three
  different points such that $1 < r = d(p,p_0)=d(q_0,p_0)$ and
  $d(p,q_0) < r-1$. Let $\beta$ be the radial projection (from $p_0$) of
  the geodesic connecting $p$ and $q_0$ into $S_r(p_0)$.  Then there is a function $b:   \R \to (0,\infty)$ such that
  $$ 
  d_{S_r(p_0)}(p,q_0) \le {\rm length}(\beta) \le 
  \left(\max_{\vert s \vert \le r} b(s)\right) d(p,q_0), 
  $$ 
  where $d_{S_r(p_0)}$ is the intrinsic distance of the sphere
  $S_r(p_0)$.
\end{prop}

\begin{proof}
  Let $\gamma: [0,d(p,q_0)] \to X$ be the geodesic connecting $p$ and
  $q_0$. We first write $\gamma$ and $\beta$ in polar coordinates,
  i.e.,
  $$ \gamma(t) = \exp_{p_0} (d(t)v(t)), \quad \beta(t) = \exp_{p_0} (rv(t)) $$
  with $d(t) = d(p_0,\gamma(t)) > 1$ and $v: [0,d(p,q_0)] \to
  S_{p_0}X$. Then we have
  $$ \gamma'(t) = d'(t) c_{v(t)}'(d(t)) + A_{v(t)}(d(t))(v'(t)) $$
  and $\beta'(t) = A_{v(t)}(r)(v'(t))$ . Note that the lower bound on sectional curvature yields the existence of a function
  $b: \R \to [0,\infty)$ such that for all $v \in SX $ and $r \ge 1$ we have  $\|S_{v,r}(t)\| \le b(t) $ 
  (see proof of Lemma 2.16 in \cite{Kn1}). Using $A_v(r) A_v^{-1}(x) =
  S_{v,x}(x-r)$ we therefore obtain 
  
  \begin{eqnarray*}
    \Vert \beta'(t) \Vert &=& 
    \Vert A_{v(t)}(r) v'(t) \Vert   
    \le \Vert A_{v(t)}(r) A_{v(t)}^{-1}(d(t)) \Vert \cdot 
    \Vert  A_{v(t)}(d(t)) v'(t) \Vert \\
    &=& \Vert S_{v(t),d(t)}(d(t)-r) \Vert \cdot 
    \Vert  A_{v(t)}(d(t)) v'(t) \Vert \\
    &\le& b(d(t)-r) \sqrt{\Vert  A_{v(t)}(d(t)) v'(t) \Vert^2 + 
    \Vert d'(t) c_{v(t)}'(d(t) \Vert^2 } \\
    &\le& \left(\max_{\vert s \vert \le r} b(s)\right) \Vert \gamma'(t) \Vert.
  \end{eqnarray*}
  The last inequality above follows from $r - d(p,q_0) \le d(t) =
  d(p_0,\gamma(t)) \le r + d(p,q_0)$ and $d(p. q_0) \le r$.
\end{proof}

For the next result, we need to introduce for every $v \in SX$ and $ r
> 0$ the function $b_{v,r}(p) = d(c_v(r),p)-r$ and the
 Busemann function $b_v(p) = \lim_{r \to \infty} b_{v,r}(p)$. Since we also use a uniform bound on
 on the norm of Jacobi tensor $S_{v,r}(t)$ for all $t \ge 0$ and $r \ge 1$ as has been derived in 
\cite[Cor. 2.6]{KnPe2} we need the assumption on $X$ made at the beginning of this section.
\begin{cor} \label{cor:vwclose} Let $p,q \in X$, $r > 2d(p,q)+1$, $v
  \in S_pX$ and $w = - \grad b_{v,r} (q) \in S_qX$. Then there exists
  a constant $C = C(r) > 0$ such that
  $$ d(c_v(t),c_w(t)) \le (1 + 2 C e^{-\frac{\rho}{2}t}) d(p,q) \quad 
  \text{for all $0 \le t \le r$.}  $$
\end{cor}

\begin{proof}
  Let $p_0 = c_v(r)$, $q_0 = c_w(d(p_0,q)-r) \in S_r(p_0)$ and $w_0 =
  c_w'(d(p_0,q)-r)$. Then we have 
  $$ d(p,q_0) \le d(p,q) + d(q,q_0) \le 2 d(p,q) < r-1. $$
  Let $\beta: [0,1] \to S_r(p_0)$ be the intrinsic geodesic in
  $S_r(p_0)$ connecting $p$ and $q_0$. Let $d_{p_0}(x) = d(p_0,x)$ and
  $N(x) = - \grad d_{p_0}(x)$ for $x \neq p_0$. Let $\beta_t: [0,1]
  \to S_{r-t}(p_0)$ defined by $\beta_t(s) = c_{N(\beta(s))}(t)$ for
  $t \in [0,r)$. Then $\beta_t'(s) =
  S_{N(\beta(s)),r}(t)(\beta'(s))_t$, which implies, using
  \cite[Cor. 2.6]{KnPe2},
  $$ \Vert \beta_t'(s) \Vert \le \Vert S_{N(\beta(s)),r}(t) \Vert \cdot 
  \Vert \beta'(s) \Vert \le a_2 e^{-\frac{\rho}{2}t} \Vert \beta'(s) \Vert. $$
  Consequently,
  $$ d(c_v(t),c_{w_0}(t)) \le {\rm length}(\beta_t) \le 
  a_2 e^{-\frac{\rho}{2}t} d_{S_r(p_0)}(p,q_0) \le C e^{-\frac{\rho}{2}t} d(p,q_0)$$
  with $C = a_2 \max_{\vert s \vert \le r} b(s)$, using
  Proposition \ref{prop:inexdist}. This implies
  \begin{eqnarray*} 
  d(c_v(t),c_w(t)) &\le& d(c_v(t),c_{w_0}(t)) + d(c_{w_0}(t),c_w(t)) \\
  &\le& C e^{-\frac{\rho}{2}t} d(p,q_0) + d(q,q_0)\\ 
  &\le& (1 + 2 C e^{-\frac{\rho}{2}t}) d(p,q).
  \end{eqnarray*}
\end{proof}

Now we prove surjectivity of $\varphi_p$: Let $c: [0,\infty) \to X$ be
a geodesic ray with $w = c'(0) \in S_qX$. Let $v = - \grad b_w(p) \in
S_pX$.  Then $c_v$ is asymptotic to $c_w$ by Corollary
\ref{cor:vwclose} with $r = \infty$. Therefore $\varphi_p(v) =
c(\infty)$ and $\varphi_p$ is surjective.

We define $\overline{X} = X \cup X(\infty)$ and introduce for every $p
\in X$ the following bijective map $\bar \varphi_p: \overline{B_1(p)}
\to \overline{X}$, where $\overline{B_1(p)} \subset T_pX$ is the
closed ball of radius $1$:
$$ \bar \varphi_p(v) = \begin{cases} \varphi_p(v) & \text{if $\Vert v \Vert =1$,} \\ \exp_p\left( \frac{1}{1-\Vert v \Vert}  v\right) & \text{if $\Vert v \Vert < 1$.}\end{cases} $$
We define a topology on $\overline{X}$ such that the bijective map
$\bar \varphi_p: \overline{B_1(p)} \to \overline{X}$ is a homeomorphism.
Next we show that this topology on $\overline{X}$ does not depend on
the reference point $p$. For that we need to show that $\bar
\varphi_{p,q} = \bar \varphi_q^{-1} \circ \bar \varphi_p:
\overline{B_1(p)} \to \overline{B_1(q)}$ is a homeomorphism.

For the continuity of $\bar \varphi_{p,q}$ note first that
$$ \bar \varphi_{p,q}(v) = 
\begin{cases} - \grad b_v(q) & \text{if $\Vert v \Vert =1$,} \\
  \exp_q^{-1}\left( \exp_p\left( \frac{1}{1-\Vert v \Vert} v\right)
  \right) & \text{if $\Vert v \Vert < 1$.} \end{cases} $$ 
Let $v_n \in \overline{B_1(p)}$ such that $v_n \to v \in
\overline{B_1(p)}$. If $\Vert v \Vert < 1$, the continuity of $\bar
\varphi_{p,q}$ at $v$ follows from the continuity of the exponential
maps. If $\Vert v \Vert = 1$, it suffices to consider two cases: in
the first case we have $0 \neq \Vert v_n \Vert < 1$ for all $n$ and
$\Vert v_n \Vert \to 1$, and in the second case we have $\Vert v_n
\Vert = 1$ for all $n$. We present the prove of the first case, the
second case goes analogously: Note that we have
\begin{multline*} 
  \exp_q^{-1}\left( \exp_p\left( \frac{1}{1-\Vert v_n \Vert} v_n
    \right) \right) = \\ - \frac{d(q,c_{v_n}(f(\Vert v_n
    \Vert))}{1+d(q,c_{v_n}(f(\Vert v_n \Vert))} \grad
  b_{\frac{v_n}{\Vert v_n \Vert}, f(\Vert v_n \Vert)}(q) = w_n,
\end{multline*}
where $f(x) = \frac{x}{1-x}$. We need to show that $w_n \to -\grad
b_v(q)$.  Choose a convergent subsequence $w_{n_j} \in S_qX$ with
limit $w \in S_qX$. Then there exists a constant $a > 0$ such that for
all sufficiently large $n \in \N$
$$ d(c_{v_n}(t),c_{w_n}(t)) \le a \quad 
\text{for all $0 \le t \le f(\Vert v_n \Vert) = r_n$,} $$
by Corollary \ref{cor:vwclose}. Note that $r_n \to \infty$. This
implies that
$$ d(c_v(t),c_w(t)) \le a \quad \text{for all $t \ge 0$,} $$
i.e., $c_v$ and $c_w$ are asymptotic geodesic rays. By Corollary
\ref{cor:vwclose}, $c_v$ and $c_{-\grad_v(q)}$ are also
asymptotic. Therefore, by the injectivity of $\varphi_q$, we have $w =
-\grad_v(q)$. This finishes the proof that $\bar \varphi_{p,q}$ is a
homeomorphism.

This topology on $\overline{X}$ was first introduced for Hadamard
manifolds by Eberlein-O'Neill \cite{EON} and is called {\em cone
  topology}. The points in $X(\infty) \subset \overline{X}$ are called
{\em points at infinity}. Note that a sequence $x_n \in X$ converges
in the cone topology to a point at infinity if and only if for every
$p \in X$ we have $d(x_n,p) \to \infty$ and for every $\epsilon > 0$
there exists $n(\epsilon)$ such that $\angle_p(x_n,x_m) < \epsilon$
for all $n,m \ge n(\epsilon)$. We write ``$\angle_p(x_n,x_m) \to 0$ as
$n,m \to \infty$'' for the latter.

\section{Gromov hyperbolicity}

We start this section by introducing the Gromov product. 

\begin{dfn} \label{dfn:gromprod}
  Let $(X,d)$ be a metric space and $x_0 \in X$ a reference point. The
  {\em Gromov product} $(x \vert y)_{x_0}$ of $x,y \in X$ is defined as
  $$
  (x \vert y)_{x_0} = \frac{1}{2} (d(x,x_0) + d(y,x_0) - d(x,y))
  $$
\end{dfn}

Note that the Gromov product $(x \vert y)_{x_0}$ is non-negative, by the
triangle inequality. A metric space $(X,d)$ is called a {\em geodesic
  space}, if any two points $x,y \in X$ can be connected by a
geodesic, i.e., if there exists a curve $\sigma_{xy}: [0,d(x,y)] \to
X$ connecting $x$ and $y$, such that $d(\sigma_{xy}(s),\sigma_{xy}(t))
= |t-s|$ for all $s,t \in [0,d(x,y)]$.

\begin{dfn}
  A geodesic space $(X,d)$ is called {\em $\delta$-hyperbolic} if
  every geodesic triangle $\Delta$ is $\delta$-thin, i.e., every side
  of $\Delta$ is contained in the union of the $\delta$-neighborhoods
  of the other two sides. If a geodesic space $(X,d)$ is
  $\delta$-hyperbolic for some $\delta \ge 0$, we call $(X,d)$ a
  Gromov hyperbolic space.
\end{dfn}

Let us recall the following two general results for Gromov hyperbolic
spaces. 

\begin{prop}(see \cite[Chapter 1, Prop. 3.6]{CDP}) \label{prop:gromprod1} 
  Let $(X,d)$ be a $\delta$-hyperbolic space. Then we have for all
  $x_0,x,y,z \in X$:
  $$
  (x \vert y)_{x_0} \geq \min \{(x \vert z)_{x_0}, (y \vert z)_{x_0}\} - 8 \delta.
  $$
\end{prop}

\begin{prop}(see \cite[Chapter 3, Lem. 2.7]{CDP}) \label{prop:gromprod2}
  Let $(X,d)$ be a $\delta$-hyperbolic space. Then we have for all
  $x_0,x,y \in X$:
  $$
  (x \vert y)_{x_0} \leq d(x_0,\sigma_{xy}) \leq (x \vert y)_{x_0} + 32 \delta.
  $$
\end{prop}

Now assume that $X$ is a rank one asymptotically harmonic manifold
satisfying \eqref{eq:curvboundcond} and, therefore, a Gromov
hyperbolic space, by \cite[Thm. 1.5]{KnPe2}. We show now that two
sequences $\{ x_n \}, \{ y_n \} \subset X$ have the same limiting
behavior at infinity in the cone topology if and only if
$$
\lim_{n \to \infty} (x_n \vert y_n)_{p} = \infty. 
$$
We note that condition $\lim_{n,m \to \infty} (x_n \vert x_m)_{p} =
\infty$ is used for general Gromov hyperbolic space as a definition
for convergence to infinity (see \cite[Section 2.2]{BS}).

\begin{theorem} \label{thm:gromconv} Let $X$ be a rank one
  asymptotically harmonic manifold satisfying
  \eqref{eq:curvboundcond}. Let $p \in X$ and $\{x_n\}, \{y_n\}$ be
  two sequences in $X$. The following are equivalent.
  \begin{itemize}
  \item[(a)] We have  $d(x_n,p), d(y_n,p) \to \infty$ and $\angle_p(x_n,y_n) \to
    0$ for $n \to \infty$.
  \item[(b)] $(x_n \vert y_n)_p \to \infty$ for $n \to \infty$.
  \end{itemize}
\end{theorem}

\begin{proof}
  $(b) \Rightarrow (a)$: $X$ is $\delta$-hyperbolic for some $\delta
  \ge 0$. Let $(x_n \vert y_n)_p \to \infty$.  We know from
  Proposition \ref{prop:gromprod2} that $d(p,x_n), d(p,y_n) \ge (x_n
  \vert y_n)_p$, which shows that $d(p,x_n), d(p,y_n) \to \infty$ as $n \to
  \infty$. It remains to show that $\angle_p(x_n,y_n) \to 0$. Let
  $U_{px_n}, U_{py_n}$ be $\delta$-tubes around the geodesic arcs
  $\sigma_{px_n}$ and $\sigma_{py_n}$. Then the geodesic $\sigma_{x_n
    y_n}$ must contain a point $p_1 \in U_{px_n} \cap U_{py_n}$. We
  conclude from Proposition \ref{prop:gromprod2} that
  $$ d(p_1,p) \ge d(\sigma_{x_ny_n},p) \ge (x_n \vert y_n)_p. $$
  Let $\gamma_1$ and $\gamma_2$ be the shortest curves connecting
  $p_1$ with $\sigma_{px_n}$ and $\sigma_{py_n}$ at the points $\hat x_n$
  and $y_n'$, see Figure \ref{fig:gromthin}. Then
  $d(p_1,\hat x_n),d(p_1,y_n') \leq \delta$, which implies $d(\hat x_n,y_n') \le
  2\delta$ and
  $$ d(\hat x_n,p), d(y_n',p) \ge (x_n \vert y_n)_p - \delta. $$

  \begin{figure}[h]
  \begin{center}
    \psfrag{Upxn}{$U_{px_n}$} 
    \psfrag{Upxm}{$U_{py_n}$}
    \psfrag{xn}{$x_n$}
    \psfrag{xm}{$y_n$}
    \psfrag{p1}{$p_1$}
    \psfrag{p}{$p$}
    \psfrag{yn}{$\hat x_n$}
    \psfrag{ym}{$y_n'$}
         \includegraphics[width=10cm]{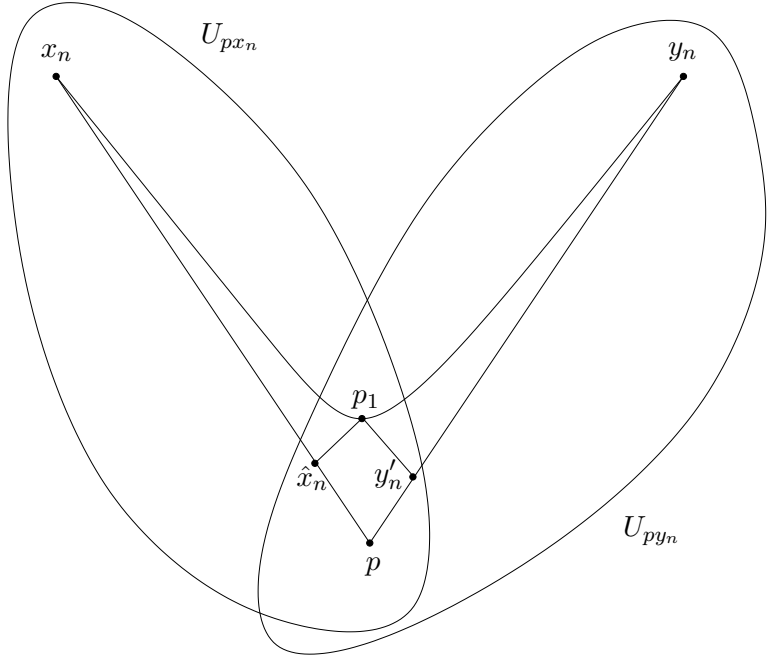}
  \end{center}
  \caption{Illustration of the proof of $(b) \Rightarrow (a)$ in Theorem
    \ref{thm:gromconv}}
  \label{fig:gromthin}
  \end{figure}

  We assume, without loss of generality, that $d(\hat x_n,p) \ge
  d(y_n',p)$. Let $\hat y_n \in \sigma_{py_n}$ be such that $d(p,\hat y_n) =
  d(p,\hat x_n)$. This implies that
  $$ d(\hat x_n,p) = d(\hat y_n,p) \ge (x_n \vert y_n)_p - \delta. $$
  Since
  $$ d(y_n',p) \le d(\hat y_n,p) = d(\hat x_n,p) \le d(y_n',p) + d(\hat x_n,y_n')
  \le d(y_n',p) + 2 \delta, $$ 
  and since $y_n', \hat y_n$ lie on the same geodesic arc $\sigma_{py_n}$,
  we have $d(y_n',\hat y_n) \le 2 \delta$. This implies that
  $$ 
  d(\hat y_n,\hat x_n) \le d(y_n',\hat x_n) + d(\hat y_n,y_n') \le 
  2\delta + 2 \delta = 4 \delta. 
  $$ 
  Using Corollary \ref{cor:uniformdiv}, we conclude that
  $$ 4 \delta \ge {\rm length}(\sigma_{\hat x_n \hat y_n}) \ge a(d(\hat x_n,p)) 
  \angle_p(x_n,y_n). $$ 
  Since $d(\hat x_n,p) \to \infty$, we also have
  $a(d(\hat x_n,p)) \to \infty$, which implies that $\angle_p(x_n,y_n) \to 0$.\\

  $(a) \Rightarrow (b)$: Assume $\angle_p(x_n,y_n) \to 0$ and
  $d(x_n,p), d(y_n,p) \to \infty$ for $n \to \infty$. For all $R > 0$,
  there exists $n_0(R) \ge 0$, such that for all $n \geq n_0(R)$:
  \begin{equation} \label{eq:smdist} d(p,x_n), d(p,y_n) \ge R \quad
    \text{and} \quad d(c_{px_n}(R),c_{py_n}(R)) \leq 1,
  \end{equation}
  since $\angle_p(x_n,y_n) \to 0$ for $n \to \infty$. Note that the
  constant $n_0(R)$ does not depend on $p$, but only on the values
  $d(p,x_n), d(p,y_n)$ and $\angle_p(x_n,y_n)$, since $X$ has a
  uniform lower curvature bound.

  We show now the following: {\em The geodesic arc $\sigma_{x_n y_n}$
    has empty intersection with the open ball $B_{R-\frac{1}{2}}(p)$
    for all $n \geq n_0(R)$.}

  If $\sigma_{x_n y_n} \cap B_R(p) = \emptyset$, there is nothing to
  prove. If $\sigma_{x_n y_n} \cap B_R(p) \neq \emptyset$, there exists a
  first $t_0 > 0$ and a last $t_1 > 0$ such that
  $$
  q_1 = \sigma_{x_n y_n}(t_0), \, q_2 = \sigma_{x_n y_n}(t_1) \in S_R(p),
  $$
  where $S_R(p)$ denotes the sphere of radius $R > 0$ around $p$ (see
  Figure \ref{fig:ballR}). Then we have
  $$
  d(q_1,q_2) = l(\sigma_{x_n y_n}) - d(x_n,q_1) - d(y_n,q_2).
  $$

  \begin{figure}[h]
  \psfrag{xn}{$x_n$} 
  \psfrag{xm}{$y_n$}
  \psfrag{q1}{$q_1$}
  \psfrag{q2}{$q_2$}
  \psfrag{SR(p)}{$S_R(p)$}
  \psfrag{spxn(R)}{$\sigma_{px_n}(R)$}
  \psfrag{spxm(R)}{$\sigma_{py_n}(R)$}
  \begin{center}
         \includegraphics[width=8cm]{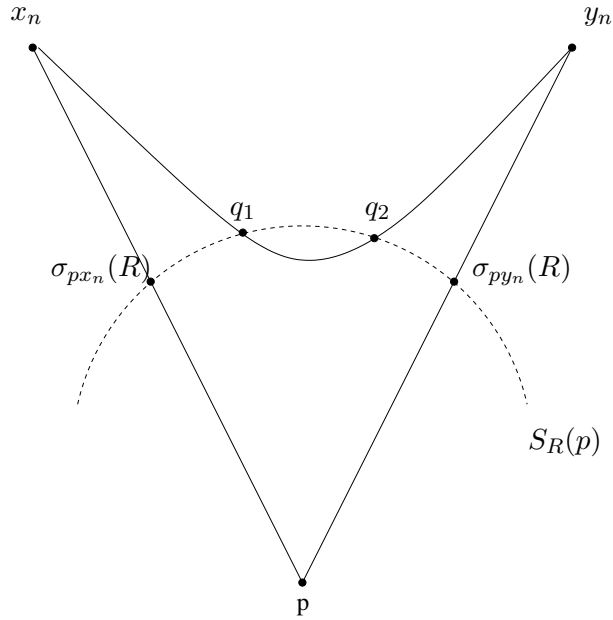}
  \end{center}
  \caption{Illustration of the proof of $(a) \Rightarrow (b)$ in Theorem
    \ref{thm:gromconv}}
  \label{fig:ballR}
  \end{figure}
  
  Using \eqref{eq:smdist}, we have
  \begin{eqnarray*}
    l(\sigma_{x_n y_n}) & \leq & d(x_n,\sigma_{px_n}(R)) + d(\sigma_{px_n}(R), 
    \sigma_{py_n}(R)) + d(\sigma_{py_n}(R),y_n)\\
    & \leq & d(x_n, \sigma_{px_n}(R)) + d(y_n,\sigma_{py_n}(R)) + 1,
  \end{eqnarray*}
  which implies that
  \begin{equation} \label{eq:dq1q2}\begin{aligned} d(q_1,q_2) &\leq
      \left(d(x_n,\sigma_{px_n}(R)) - d(x_n,q_1)\right) \\ &\qquad +
      \left( d(y_n,\sigma_{py_n}(R)) - d(y_n,q_2) \right) +
      1. \end{aligned}
  \end{equation}
  Since $d(p,x_n) = R+d(\sigma_{px_n}(R),x_n) \leq d(q_1,x_n) + R$ (by
  the triangle inequality), we obtain $d(x_n,q_1) -
  d(x_n,\sigma_{px_n}(R)) \geq 0$, and similarly $d(y_n,q_2) -
  d(y_n,\sigma_{py_n}(R)) \geq 0$. This, together with
  \eqref{eq:dq1q2} shows $d(q_1,q_2) \leq 1$. But then the geodesic
  segment of $\sigma_{x_n y_n}$ between $q_1$ and $q_2$ cannot enter
  the ball $B_{R-\frac{1}{2}}(p)$.\\

  Therefore, we have for all $n \ge n_0(R)$,
  $$
  R- \frac{1}{2} \leq d(p,\sigma_{x_n y_n}) \leq (x_n \vert y_n)_p + 32\delta, 
  $$
  using Proposition \ref{prop:gromprod2}. This shows that
  $$
  (x_n \vert y_n)_p \to \infty \quad \text{as $n \to \infty$.}
  $$
\end{proof}

\section{Visibility measures and their  Radon-Nykodym derivative}
\label{sect:visib}

Let $(X,g)$ be a rank one asymptotically harmonic manifold of
dimension $n$. The boundary $X(\infty) \subset \overline{X}$ is
homeomorphic to the sphere $S^{n-1}$ and equipped with the relative
topology of the cone topology. Moreover, we have a family of
visibility measures $\{ \mu_p \in {\mathcal M}_1(X(\infty)) \}_{p \in
  X}$, which were introduced in Definition \ref{def:vismeas}. We will
see that any two visibility measures $\mu_p, \mu_q \in {\mathcal
  M}_1(X(\infty))$ are absolutely continuous, by calculating their
Radon-Nykodym derivative via a limiting process. Similar calculations
were carried out in \cite[Section 6.1]{CaSam} for asymptotically
harmonic manifolds with pinched negative curvature.

\begin{lemma} \label{lem:bijvismap}
  For all $p, q \in X$ there exists $t(p,q) > 0$ such that for all
  $t \ge t(p,q)$ and all $v \in S_qX$ the geodesic ray $c_v: [0,
  \infty) \to X$ intersects $S_t(p)$ in a unique point $F_t(v)$ (see
  Figure \ref{fig_ftqp}). In particular, the map $F_t: S_q X \to
  S_t(p)$ is bijective.
\end{lemma}

\begin{figure}[h]
  \begin{center}
    \psfrag{SqX}{$S_qX$} 
    \psfrag{q}{$q$}
    \psfrag{v}{$v$}
    \psfrag{Ft(v)}{$F_t(v)$}
    \psfrag{St(p)}{$S_t(p)$}
    \psfrag{t}{$t$}
    \psfrag{p}{$p$}
    \includegraphics[width=8cm]{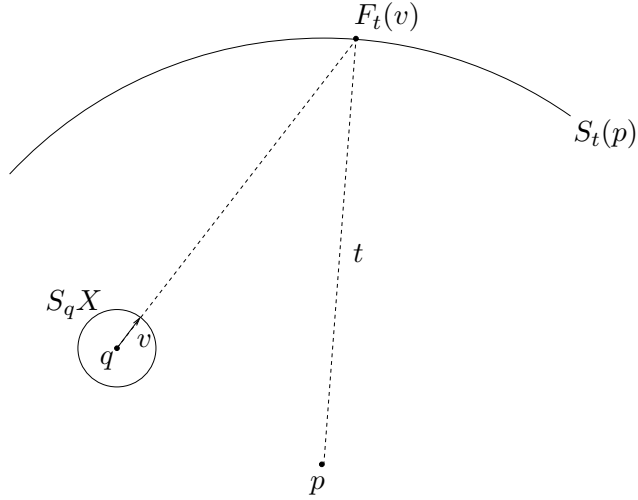}
  \end{center}
  \caption{Illustration of the map $F_t: S_qX \to S_t(p)$}
  \label{fig_ftqp}
\end{figure}

\begin{proof}
  Let $a(t)$ be as in Corollary \ref{cor:uniformdiv}. Choose $t_0$
  such that for all $t \ge t_0$ we have $2 d(p,q) \le a(t)$. Define
  $$
  t(p,q) = \max \{ d(p,q) +1, t_0 \}.
  $$
  In particular, $q$ lies in the ball of radius $t$ around $p$, for
  all $t \ge t(p,q)$, and hence for all $v \in S_qX$ the geodesic ray
  $c_v: [0, \infty) \to X$ intersects $S_t(p)$. Let $t \ge t(p,q)$,
  and assume that $q' = c_v(t_1)$ is an intersection point of
  $c_v([0,\infty))$ and $S_t(p)$ such that $c_v'(t_1)$ is either
  pointing into $B_t(p)$ or is tangent to $S_t(p)$, i.e., 
  $$
  \angle (c_v'(t_1), c_w'(t)) \ge \pi/2,
  $$
  where $w \in S_pX$ is the unique vector such that $c_w(t) = q'$. Using
  the triangle inequality we obtain
  $$
  t- d(p,q) \le t_1 \le t + d(p,q)
  $$
  Using Corollary \ref{cor:uniformdiv}, we obtain for all $s \ge 0$
  $$
  d(c_v(t_1-s), c_w(t-s)) \ge a(s) \pi/2.
  $$
  In particular for $s =t$ this yields 
  $$
  a(t) \pi/2 \le d(c_v(t_1-t), p) \le d(c_v(t_1-t),q ) + d(q,p) \le 2
  d(p,q) \le a(t),
  $$
  which is a contradiction. Hence, a second intersection point between
  the geodesic ray $c_v([0,\infty))$ and $S_t(p)$ cannot occur.
\end{proof}

\begin{prop}
  Let $(X,g)$ be a complete, simply connected noncompact manifold
  without conjugate points and $p,q \in X$. Consider the map $F_t: S_q
  X \to S_t(p)$, where $F_t(v)$ is the first intersection point of the
  geodesic ray $c_v: [0,\infty) \to X$ with $S_t(p)$. If $q$ is
  contained in the ball of radius $t$ about $p$, this map is well
  defined. Then the Jacobian of $F_t$ is given by
  \begin{equation} \label{eq:jacFt}
  \Jac F_t(v) = \frac{\det A_v(d(q,F_t(v)))}{\langle N_p(F_t(v)),N_q(F_t(v)) 
  \rangle},
  \end{equation}
  where $N_x(y) = (\grad d_x)(y)$ and $d_x(y) = d(x,y)$.
\end{prop}

Note that \eqref{eq:jacFt} agrees with \cite[(6.3)]{CaSam}. For
convenience of the readers, we provide our own proof of this formula.

\begin{proof}
  Choose a curve $\gamma : (- \epsilon, \epsilon) \to S_q X$ with
  $\gamma(0) = v \in S_q X$. Then
  $$
  F_t(\gamma(s)) = \exp_q (d(q, F_t(\gamma(s))) \cdot \gamma(s)),
  $$
  and, using the chain rule and the product rule,
  \begin{multline*}
    D F_t(v)(\gamma'(0)) = \\
    D \exp_q (d(q,F_t(v)) \cdot v)(\langle N_q(F_t(v)), D F_t(v)
    \gamma'(0) \rangle v + d(q,F_\gamma(v)) \cdot
    \gamma'(0)).
  \end{multline*}
  Note that $\gamma'(0) \perp v$. We have
  $$
  D \exp_q(t v)(t w) = Y(t)(w) = J(t),
  $$
  where $Y$ is the Jacobi tensor along $c_v$ with $Y(0) = 0$ and
  $Y'(0) = \id$, and therefore $J$ is a Jacobi field along $c$
  satisfying $J(0) =0$ and $J'(0) =w$. Note that $Y$ and $A_v$ are
  related by $A_v = Y \Big\vert_{(c_v')^\perp}$. In particular, we
  have $D \exp_q(tv)(tv) = t c_v'(t)$. This yields
  \begin{multline*}
  $$ D F_t(v)(\gamma'(0)) \\ 
  =  \langle N_q(F_t(v)), D F_t(v) \gamma'(0) \rangle\; 
    c_v' (d(q,F_t(v))) + A_v(d(q,F_t(v)))( \gamma'(0)).
  \end{multline*}
 
  Consequently,
  \begin{multline} \label{eq:jacid1}
  D F_t(v) (\gamma'(0)) = \\
  \langle N_q(F_t(v)), D F_t(v) \gamma'(0) \rangle N_q(F_t(v)) + 
  A_v(d(q,F_t(v))) (\gamma'(0)).
  \end{multline}
  Next, we introduce the map
  \begin{eqnarray*}
  L_x: N_p(x)^\perp &\to& N_q(x)^\perp,\\
  L_x(w) &=& w - \langle w, N_q(x) \rangle N_q(x).
  \end{eqnarray*}
  Then \eqref{eq:jacid1} can be rewritten as 
  \begin{equation}  \label{eq:jacid2}
  L_{F_t(v)} \circ DF_t(v) = A_v(d(q,F_t(v))). 
  \end{equation}

  To finish the proof of the above Proposition, we need the following
  lemma.

  \begin{lemma}
    $\Jac L_x = | \langle N_p(x), N_q(x) \rangle |$.
  \end{lemma}

  \begin{proof}
    Consider
    $$
    N_p(x)^\perp \cap N_q(x)^\perp = \{w \in T_x X \mid\; 
    \langle w,N_p(x) \rangle = 0 \; \; \text{and} \; \; 
    \langle w,N_q(x) \rangle = 0 \}.
    $$
    Then $N_p(x)^\perp \cap N_q(x)^\perp$ has co-dimension one in
    $N_p(x)^\perp$ and $L_x$ is the identity on $N_p(x)^\perp \cap
    N_q(x)^\perp$. Let
    $$
    w_0 = N_q(x) - \langle N_q(x), N_p(x) \rangle N_p(x) \in N_p(x)^\perp.
    $$
    The vector $w_0$ is orthogonal to $N_p(x)^\perp \cap
    N_q(x)^\perp$ since for all $w \in N_p(x)^\perp \cap
    N_q(x)^\perp$ we have $\langle w,N_p(x) \rangle = 0$ and $\langle
    w,N_q(x) \rangle = 0$, and therefore
    $$
    \langle w,w_0 \rangle = \langle \underbrace{w,N_q(x)}_{= 0} \rangle - 
    \langle N_q(x), N_p(x) \rangle \langle 
    \underbrace{w, N_p(x)}_{= 0} \rangle = 0.
    $$
    Moreover, $L_x w_0$ is also orthogonal to $N_p(x)^\perp \cap
    N_q(x)^\perp$:
    \begin{eqnarray*}
      L_x w_0 &=& w_0 - \langle w_0, N_q(x) \rangle N_q(x)\\
      &=& \langle N_p(x), N_q(x) \rangle (\langle N_p(x), N_q(x) \rangle N_q(x) 
      - N_p(x)),
    \end{eqnarray*}
    and consequently $\langle w,L_x w_0 \rangle = 0$ for all $w$
    satisfying $\langle w,N_p(x) \rangle= \langle w,N_q(x) \rangle =
    0$. This shows that
    $$
    \Jac L_x = \frac{||L_x w_0||}{||w_0||}.
    $$
    Since
    $$ \|L_x w_0\|^2 = \langle N_p(x),N_q(x) \rangle^2 
    (1 - \langle N_p(x),N_q(x) \rangle^2) $$
    and
    $$ \|w_0\|^2 = 1 - \langle N_p(x),N_q(x) \rangle^2, $$
    we obtain
    \begin{eqnarray*}
    \Jac L_x &= &\left( \frac{\langle N_p(x),N_q(x) \rangle^2 
    (1 - \langle N_p(x),N_q(x) \rangle^2)}{1 - \langle N_p(x),N_q(x) \rangle^2} 
    \right)^{1/2}\\
    &= &| \langle N_p(x),N_q(x) \rangle |,
    \end{eqnarray*}
    which yields the lemma.
  \end{proof}

  Finally, \eqref{eq:jacid2} implies that
  $$
  \Jac F_t(v) = \frac{\det A_v(d(q,F_t(v)))}{\Jac\; L_{F_t}(v)} =
  \frac{\det A_v (d(q,F_t(v)))}{\langle N_p(F_t(v)),N_q(F_t(v))
    \rangle},
  $$
  finishing the proof of the proposition.
\end{proof}

\begin{figure}[h]
  \begin{center}
    \psfrag{SqX}{$S_qX$} 
    \psfrag{q}{$q$}
    \psfrag{v}{$v$}
    \psfrag{Ft(v)}{$F_t(v)$}
    \psfrag{St(p)}{$S_t(p)$}
    \psfrag{t}{$t$}
    \psfrag{p}{$p$}
    \psfrag{Bt(v)}{$B_t(v)$}
    \includegraphics[width=8cm]{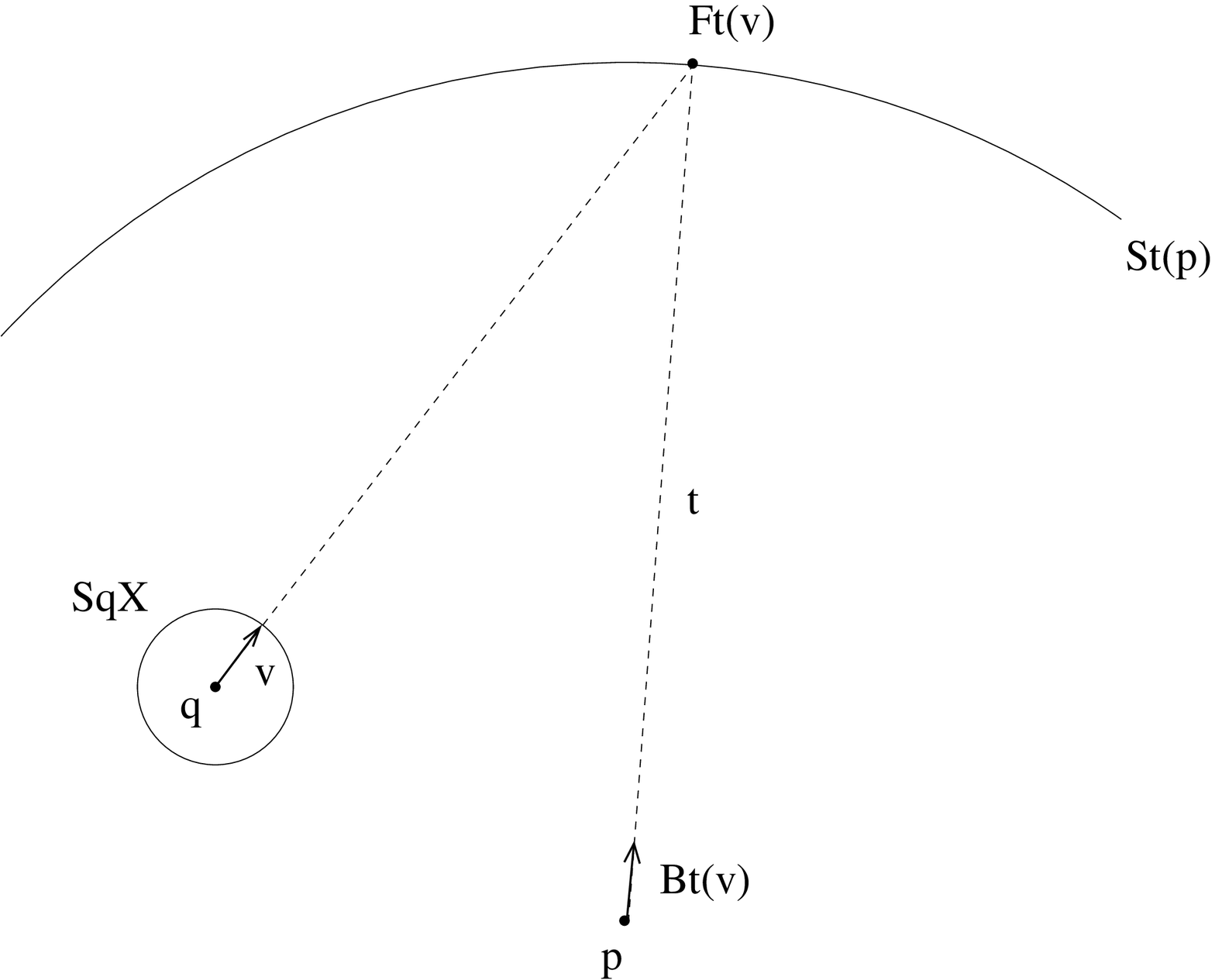}
  \end{center}
  \caption{Illustration of the map $B_t: S_qX \to S_pX$}
  \label{fig_btv}
\end{figure}

\begin{cor} \label{cor:JacBt} Let $(X,g)$ be a complete, simply
  connected noncompact manifold without conjugate points and $p,q \in
  X$. Let $B_t: S_q X \to S_p X, v \mapsto \frac{1}{t} \exp_p^{-1}
  \circ F_t(v)$ (see Figure \ref{fig_btv}). Then we have
  $$ \Jac B_t(v) = \frac{\det A_v(d(q,F_t(v)))}{\det A_u(t)} 
  \cdot \frac{1}{\langle N_p(F_t(v)),N_q(F_t(v)) \rangle},
  $$
  where $u = B_t(v)$.
\end{cor}

\begin{proof}
  Let $u \in S_p X$. Then $D \exp_p(tu) : u^\perp \to T_{\exp_p(tu)}
  S_t(p)$ is given by $D \exp_p(tu) (w ) = \frac{1}{t} A_u(t)(w)$, and
  therefore with $u = B_t(v)$,
  \begin{eqnarray*}
  \Jac B_t(v) &=& \frac{1}{\det A_u(t)} \cdot \Jac F_t(v)\\ 
  &=& \frac{\det A_v(d(q,F_t(v)))}{\det A_u(t)} \cdot 
  \frac{1}{\langle N_p(F_t(v)),N_q(F_t(v)) \rangle}.
  \end{eqnarray*}
 \end{proof}

 From now on, $(X,g)$ denotes a rank one asymptotically harmonic
 manifold satisfying \eqref{eq:curvboundcond} with $n={\rm
   dim}(X)$. Let $f \in C(X(\infty))$. We know from Lemma
 \ref{lem:bijvismap} that $B_t: S_qX \to S_pX$ is a bijection, for $t
 > 0$ large enough. Then we have with $f_1 = f \circ \varphi_p$:
\begin{eqnarray*}
  \int_{X(\infty)} f(\xi)\, d\mu_p(\xi) &=& 
  \frac{1}{\omega_n} \int_{S_pX} f_1(w)\; d\theta_p(w) \\
  &=& \frac{1}{\omega_n} \int_{S_qX} (f_1 \circ B_t)(v) (\Jac B_t)(v)\;
  d\theta_q(v).
\end{eqnarray*}
We will show that 
\begin{itemize}
\item[(i)] $\lim_{t \to \infty} B_t = (\varphi_p)^{-1} \circ
  \varphi_q$,
\item[(ii)] There exist constants $t_0 > 0$ and $C > 0$ such that
$$| \Jac B_t(v) | \le C \quad \forall\, v \in S_qX, \, t \ge t_0. $$  
\item[(iii)] We have, for all $v \in S_qX$,
$$ \lim_{t \to \infty} \Jac B_t(v) = e^{-h b_v(p)}, $$
\end{itemize}
where $b_v$ is the Busemann function introduced in
\eqref{eq:Busemann}. Having these facts, we conclude with Lebesgue's
dominated convergence that
\begin{eqnarray*}
  \int_{X(\infty)} f(\xi)\, d\mu_p(\xi) &=& 
  \lim_{t \to \infty} \frac{1}{\omega_n} \int_{S_qX} (f_1 \circ B_t)(v) 
  (\Jac B_t)(v)\; d\theta_q(v) \\
  &=& \frac{1}{\omega_n} \int_{S_qX} (f \circ \varphi_q)(v) e^{-h b_v(p)}
  \; d\theta_q(v) \\
  &=& \int_{X(\infty)} f(\xi) e^{-h b_{q,\xi}(p)}\, d\mu_q(\xi),  
\end{eqnarray*} 
with $b_{q,\xi} = b_v$ with $\xi = c_v(\infty)$ and $v \in S_qX$. This proves
Theorem \ref{thm:radon-nykodym} from the Introduction:

\begin{thm1}  Let $(X,g)$  be a  rank one  asymptotically harmonic
  manifold  satisfying \eqref{eq:curvboundcond}.  Let $(\mu_p)_{p  \in
    X}$ be  the associated family  of visibility measures.  Then these
  measures are pairwise absolutely continuous and we have
  $$ \frac{d\mu_p}{d\mu_q}(\xi) = e^{-h b_{q,\xi}(p)}. $$
\end{thm1}

It remains to prove properties (i), (ii) and (iii) listed above.

\bigskip

{\bf Proof of (i):} Let $t_n \to \infty$ and $s_n \ge 0$, $w_n =
B_{t_n}(v) \in S_pX$ such that $y_n = \exp_q(s_n v) = \exp_p(t_n
w_n)$. We obviously have $s_n \to \infty$ and $y_n \to
\varphi_q(v)$. Let $w_{n_j}$ be a convergent subsequence of $w_n =
B_{t_n}(v)$ with limit $w \in S_pX$. Then we have $y_{n_j} \to
\varphi_p(w)$ and
$$ \varphi_q(v) = \varphi_p(w). $$
This shows that $\lim_{n \to \infty} B_{t_n}(v) = (\varphi_p)^{-1}
\circ \varphi_q(v)$. \qed

\bigskip

For the proof of (ii), we need the following lemma:

\begin{lemma} \label{lem:NpNq} 
  For every $\epsilon > 0$, there exists $t_0 > 0$ such that we have
  for all $v \in S_qX$
  $$
  |\langle N_p(F_t(v)),N_q(F_t(v)) \rangle - 1| < \epsilon \quad
  \forall\; t \ge t_0.
  $$
\end{lemma}

\begin{proof}
  This is an easy consequence of Corollary \ref{cor:uniformdiv}.
\end{proof}

\bigskip

{\bf Proof of (ii):} We start with the formula (see
\cite[p. 676]{Kn2})
$$ \det A_v(t) = \frac{\det U_v(t)}{\det(U_v'(0)-S_{v,t}'(0))} = 
\frac{e^{ht}}{\det(U_v'(0)-S_{v,t}'(0))}. $$ 
Then
$$ \frac{\det A_v(d(q,F_t(v))}{\det A_{u_t}(t)} = 
\frac{e^{h d(q,F_t(v))}\det(U_{u_t'}(0) - S_{u_t,t}'(0))}
{\det(U_v'(0)-S_{v,d(q,F_t(v))}'(0)) e^{ht}}, $$
where $u_t = B_t(v) \in S_pX$.

Let $\epsilon > 0$ be chosen. Since $\det(U_v'(0) - S_{v,t}'(0))$
converges monotonically to a universal constant $A > 0$ (see
\cite[Theorem 1.3]{KnPe2} and use the fact that $X$ is rank one), we
conclude with Dini that the convergence is uniformly on compact sets.
Therefore, there exists $t_0 \ge 0$ such that $A \le \det (U_w'(0) -
S_{w,t}'(0)) \le A + \epsilon$ for all $w \in S_pX \cup S_qX$ and $t
\ge t_0$. Using Lemma \ref{lem:NpNq} and increasing $t_0 > 0$, if
necesary, we can also assume that
$$ \langle N_p(F_t(v)),N_q(F_t(v)) \rangle \ge \frac{1}{2} $$
for all $t \ge t_0$. Since $d(q,F_t(v)) \le t + d(p,q)$, we conclude
from Corollary \ref{cor:JacBt} for all $t \ge t_0$ and all $v \in
S_qX$,
$$ \hspace{4cm} |\Jac B_t(v)| \le 2 \frac{A+\epsilon}{A} e^{h d(p,q)}. 
\hspace{3.3cm} \qed $$ 

\bigskip

{\bf Proof of (iii):} This is a immediate consequence of Lemma
\ref{lem:NpNq} and the following Lemma:

\begin{lemma}
  Using the notation above we have that
  $$
  \lim_{t \to \infty} \frac{\det A_v(d(q,F_t(v)))}{\det A_{u_t}(t)} = e^{-hb_v(p)},
  $$
  where $u_t = B_t(v)$.
\end{lemma}

\begin{proof}
  We recall that
  $$ \frac{\det A_v(d(q,F_t(v)))}{\det A_{u_t}(t)} = 
  e^{h (d(q,F_t(v))-t)}\frac{\det(U_{u_t'}(0) -
  S_{u_t,t}'(0))}{\det(U_v'(0)-S_{v,d(q,F_t(v))}'(0))} $$ 
  and
  $$ \frac{\det(U_{u_t'}(0) - S_{u_t,t}'(0))}{\det(U_v'(0)-S_{v,d(q,F_t(v))}'(0))} 
  \to \frac{A}{A} = 1. $$ 
  Now the lemma follows from
  \begin{eqnarray*}
  \lim_{t \to \infty} d(q,F_t(v)) - t &=& 
  \lim_{t \to \infty} d(q,F_t(v)) - d(p,F_t(v)) \\
  &=& \lim_{s \to \infty} d(q,c_v(s)) - d(p,c_v(s)) \\
  &=& \lim_{s \to \infty} s - d(p,c_v(s)) = - b_v(p).
  \end{eqnarray*}
\end{proof}

\begin{remark}
  Theorem \ref{thm:radon-nykodym} has an analogue for simply
  connected, noncompact harmonic manifolds $(X,g)$ without the rank
  one condition and replacing the geometric boundary $X(\infty)$ by
  the Busemann boundary (see \cite[Theorem 12.6]{KnPe1}).  There, we
  have $\det A_v(t) = f(t)$ for all $v \in SX$, where $f(t)$ is the
  volume density function, and $f(t)$ is an exponential
  polynomial. Moreover, the uniform divergence of geodesics (Corollary
  \ref{cor:uniformdiv}) holds there without the rank one
  condition. These results are not known for general asymptotically
  harmonic manifolds.
\end{remark}

\section{Solution of the Dirichlet problem at infinity}
\label{chp:dirichprob}

Since rank one asymptotically harmonic manifolds $(X,g)$ satisfying
\eqref{eq:curvboundcond} are Gromov hyperbolic with positive Cheeger
constant (see \cite{KnPe2}), general results of Ancona yield that the
Martin boundary and the geometric boundary coincide
(\cite[Th{\'e}or{\`e}me 6.2]{Anc2}) and that the Dirichlet problem at
infinity has a solution (\cite[Th{\'e}or{\`e}me 6.7]{Anc2}). In this
section we give an alternative direct proof that the Dirichlet problem
at infinity has a solution for these manifolds by providing a concrete
integral formula of the solution using the visibility
measures. Moreover, this shows that the visibility measures coincide
with the harmonic measures on $X(\infty)$.

A crucial step for our result of this section is to show that $\lim_{x
  \to \xi} \mu_x = \delta_\xi$, where $\delta_\xi$ is the
$\delta$-distribution at $\xi$. This abstract condition will follow
from the next proposition. To state it, we introduce for $v_0 \in
S_pX$ and $\delta > 0$ the cone
$$
C(v_0, \delta) = \{c_v(t) \mid\; t \in [0,\infty],\, \angle(v_0,v) \leq 
\delta\}.
$$
Note that the set of all truncated cones $C(v_0,\delta) \cap B_R(p)^c$
together with all open balls $B_r(q)$ define a basis of the cone
topology of the geometric compactification $\bar X$.
 
\begin{prop} \label{prop:horoincone} Let $(X,g)$ be a rank one
  asymptotically harmonic manifold satisfying
  \eqref{eq:curvboundcond}. Let $p \in X$ and $\delta > 0$. Then there
  exists a constant $C_1 = C_1(\delta) > 0$ such that for all $v \in
  S_pX$
  $$
  b_{v}(q) \geq d(p,q)-C_1 \quad \text{for all $q \in X \backslash 
  C(v, \delta)$.} 
  $$
\end{prop}

\begin{proof}
  Let $p \in X$ and $\delta > 0$ be given. Then there exists a
  constant $C_1 > 0$ such that
  \begin{equation} \label{eq:gromprodest} 0 \leq 2(c_v(t) \vert q)_p
    \leq C_1 \quad \forall\; t \geq 0 \quad \forall\; v \in S_pX \quad
    \forall\; q \in X \backslash C(v, \delta),
  \end{equation}
  where $( \cdot \vert \cdot)_p$ is the Gromov product introduced in
  Definition \ref{dfn:gromprod}. If this were false, then we could
  find sequences $t_n \geq 0$, $v_n \in S_pX$ and $q_n \in X
  \backslash C(v_n,\delta)$ such that
  $$
  (c_{v_n}(t_n) \vert q_n)_p \to \infty.
  $$
  Let $q_n = c_{w_n}(r_n)$ with $w_n \in S_pX$ and $r_n=d(q_n,p)$.
  This would mean, by Theorem \ref{thm:gromconv}, that $d(p,q_n) \to
  \infty$ and $\angle_p(v_n,w_n) \to 0$, which is a contradiction to
  $q_n \in X \backslash C(v_n,\delta)$.

  \eqref{eq:gromprodest} means that
  $$
  d(p,q) - (d(c_v(t),q)-t) \leq C_1 \quad \forall\; t \ge 0.
  $$
  Taking the limit $t \to \infty$, we obtain
  $$
  d(p,q) - b_v(q) = d(p,q) - \lim\limits_{t \to \infty}
  (d(c_v(t),q)-t) \leq C_1,
  $$
  finishing the proof.
\end{proof}

\begin{remark} The statement of the proposition includes the fact that
  any horoball $\mathcal H$, centered at $\xi = c_{v}(\infty) \in
  X(\infty)$, ends up inside any given cone $C(v,\delta)$, when being
  translated to a horoball $\widetilde {\mathcal H}$ along the stable
  direction (see the illustration in Figure
  \ref{fig:horoincone}). (Note that the horoballs centered at $\xi$
  can be described by $\{ q \in X \mid b_v(q) \le -C \}$, and that
  these horoballs become smaller and shrink towards the limit point
  $\xi$, as $C \in \R$ increases to infinity.)

  \begin{figure}[h]
  \psfrag{H}{$\mathcal H$} 
  \psfrag{Ht}{$\widetilde{\mathcal H}$}
  \psfrag{v0}{$v$}
  \psfrag{Cv0d}{$C(v,\delta)$}
  \psfrag{X(inf)}{$X(\infty)$}
  \psfrag{xi}{$\xi$}
  \begin{center}
         \includegraphics[width=8cm]{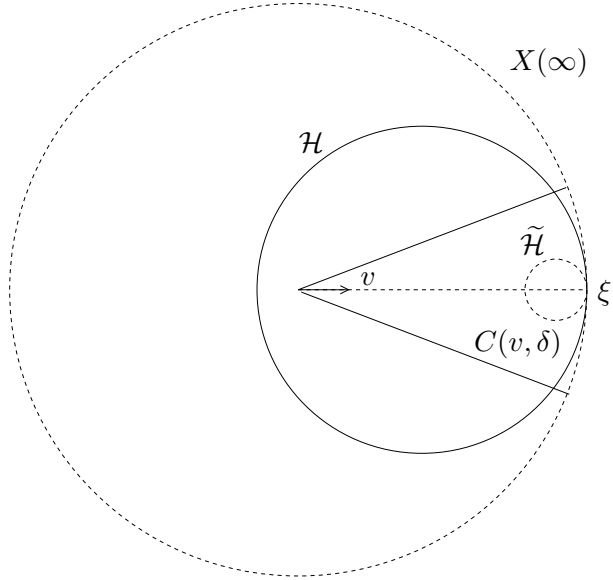}
  \end{center}
  \caption{Geometric property necessary for the solution of the Dirichlet
    problem at infinity}
  \label{fig:horoincone}
  \end{figure}
\end{remark}

\begin{remark}
  Proposition \ref{prop:horoincone} does not hold if $(X,g)$ is the
  Euclidean space. In this case, every horoball is a halfspace, which
  lies never inside a given cone.
\end{remark}

\medskip

Now we state our main result of this section, namely, the solution of
the Dirichlet problem at infinity for rank one asymptotically harmonic
manifolds satisfying \eqref{eq:curvboundcond} via an explicit integral
formula involving the visibility measures (see Theorem
\ref{thm:dirichprob} from the Introduction).

\begin{thm2} Let $(X,g)$ be a rank one asymptotically harmonic
  manifold satisfying \eqref{eq:curvboundcond}. Let $f: X(\infty) \to
  {\R}$ be a continuous function. Then there exists a unique harmonic
  function $H_f: X \to {\R}$ such that
  \begin{equation}
  \lim\limits_{x \to \xi} H_f(x) = f(\xi).
  \end{equation}
  Moreover, $H_f$ has the following integral presentation:
  $$
  H_f(x) = \int\limits_{X(\infty)} f(\xi) d \mu_x(\xi),
  $$
  where $\{\mu_x\}_{x \in X} \subset {\mathcal M}_1(X(\infty))$ are
  the visibility probability measures.
\end{thm2}

\begin{proof}

  \noindent 
  {(a)} We show first that $\int\limits_{X(\infty)} f(\xi) d
  \mu_x(\xi)$ is a harmonic function. Let $p \in X$. Then
   $$ \Delta_x \int\limits_{X(\infty)} f(\xi) d \mu_x(\xi) = \Delta_x
  \int\limits_{X(\infty)} f(\xi) e^{-h b_{p,\xi}(x)} d \mu_p(\xi). $$
  Let $K \subset X$ be a compact set. Then $x \mapsto f(\xi) e^{-h
    b_{p,\xi}(x)}$ is bounded for all $x \in K$ and all $\xi \in
  X(\infty)$, because of $|b_{p,\xi}(x)| \leq d(p,x)$. Moreover
  $\Delta_x f(\xi) e^{-h b_{p,\xi}(x)} = 0$ and $b_{p,\xi}(\cdot)$ is
  smooth, because of $\Delta_x b_{p,\xi} = h$. Therefore,
  $$
  \Delta_x \int\limits_{X(\infty)} f(\xi) d \mu_x(\xi) =
  \int\limits_{X(\infty)} f(\xi) \underbrace{\Delta_x e^{- h
      b_{p,\xi}(x)}}_{= 0} d \mu_p(\xi) = 0.
  $$

  \smallskip

  \noindent 
  {(b)} Now we prove
  $$
  \lim\limits_{x \to \xi_0} \int\limits_{X(\infty)} f(\xi) d
  \mu_x(\xi) = f(\xi_0).
  $$
  Let $\xi_0 = c_{v_0}(\infty)$ with $v_0 \in S_p X$. Without loss of
  generality, we can assume that $f(\xi_0) = 0$ (by subtracting a
  constant if necessary). Let $\epsilon > 0$ be given. Then there
  exists $\delta > 0$, such that
  $$
  |\; f(c_v(\infty))\; | \leq \epsilon \quad \forall\; v \in S_pX \;
  \text{with}\; \angle_p(v_0,v) \leq \delta.
  $$
  We split the integral representing $H_f(x)$ in the following way:
  \begin{multline*}
    \omega_n |H_f(x)| \leq \left| \int_{S_pX\; \backslash\; \{v\;
        \mid\; \angle(v_0,v) \leq \delta\}} f(c_v(\infty))\;
      e^{-h b_v(x)}\; d\theta_p(v)\right| + \\
    \left |\int_{\{v\; \mid\; \angle(v_0,v) \leq \delta\}}
      f(c_v(\infty))\; e^{-h b_v(x)}\; d \theta_p(v) \right|.
  \end{multline*}
  Now, using Proposition \ref{prop:horoincone}, we obtain for all $x
  \in C(v_0,\delta/2)$ and $C_1 = C_1(\delta/2)$
  \begin{multline*}
    \omega_n |H_f(x)| \leq \Vert f \Vert_\infty \int_{S_pX\;
      \backslash\; \{v\; \mid\; \angle(v_0,v) \leq \delta\}}
    e^{-h(d(p,x)-C_1)}\; d \theta_p(v) + \\
    \epsilon \int_{\{v\; \mid\; \angle(v_o,v) \leq \delta\}}
    e^{-h b_v(x)}\; d \theta_p(v) \leq\\
    \Vert f \Vert_\infty\; \omega_n\; e^{hC_1}\; e^{-hd(p,x)} +
    \epsilon \underbrace{\int_{S_pX} e^{-hb_v(x)}d
      \theta_p(v)}_{=\int_{S_xX}
      d \theta_x(v) = \omega_n} \leq\\
    \omega_n \left(\epsilon + \Vert f \Vert_\infty\; e^{hC_1}\;
      e^{-hd(p,x)}\right).
  \end{multline*}
  Let $x_n = c_{v_n}(r_n)$ with $v_n \in S_pX$ and $r_n \ge 0$ be a
  sequence converging to $\xi_0 \in X(\infty)$. Then we have $r_n =
  d(p,x_n) \to \infty$ and $\angle_p (v_0,v_n) \to 0$. Since $\epsilon
  > 0$ was arbitrary, the above estimate shows that
  $$
  H_f(x) \to 0\quad \text{for}\, x \to \xi_0.
  $$

  \smallskip

  \noindent
  {(c)} Uniqueness of the solution follows from the maximum
  principle.
\end{proof}

\begin{remark} The above considerations show that rank one
  asymptotically harmonic manifolds $(X,g)$ with reference point $x_0
  \in X$ satisfying \eqref{eq:curvboundcond} admit Poisson kernels of
  the form $P(x,\xi) = e^{-h b_{x_0,\xi}(x)}$. 

  These Poisson kernels can be used to define a map $\varphi: X \ni x
  \to P(x,\xi) d\mu_{x_0}(\xi) \in {\mathcal P}(X(\infty))$, where
  ${\mathcal P}(X(\infty))$ is the space of all probability measures
  on $\partial X$ which are absolutely continuous to
  $\mu_{x_0}$. ${\mathcal P}(X(\infty))$ carries a natural Riemannian
  metric $G$, called the {\em Fisher-Information metric} (see \cite{Fr}
  or \cite{ItSa1} for more details). The following was proved in
  \cite[Prop. 1]{ItSa1} for homogeneous Hadamard manifolds of
  dimension $n$: if $(X,g)$ admits Poisson kernels of the form
  $P(x,\xi) = e^{-c b_{x_0,\xi}(x)}$ with $c > 0$, then the Poisson
  kernel map $\varphi: X \to {\mathcal P}(X(\infty))$ satisfies
  $\varphi^* G = \frac{c^2}{n} g$, i.e., that $\varphi$ is a
  homothety. Examples of such spaces are rank one symmetric spaces of
  non-compact type and Damek-Ricci spaces. Conversely, the following
  was shown in \cite[Thm 1.3]{ItSa2}: If $(X,g)$ is an $n$-dimensional
  Hadamard manifold admitting a Poisson kernel map $\varphi: X \to
  {\mathcal P}(X(\infty))$, which is both a {\em homothety} with
  constant $\frac{c^2}{n}$, $c > 0$ and {\em minimal}, then $(X,g)$ is
  necessarily asymptotic harmonic with horospheres of mean curvature
  $c$. These results provide an interesting characterization of
  asymptotic harmonic manifolds via the Poisson kernel map.
\end{remark}

\section{Polynomial volume growth of horospheres}
\label{chp:horopoly}

Let $(X,g)$ be a rank one asymptotically harmonic manifold satisfying
\eqref{eq:curvboundcond}. Let $W^s(v) \subset SX$ be a strong stable
manifold through $v \in SX$. Its projection ${\mathcal H}_v = \pi
W^s(v) \subset X$ is a horosphere orthogonal to $v$. Let $p = \pi(v)$.
Consider a curve
$$
\beta : [0,1] \to \HH_v
$$
with ${\rm length}(\beta) \le r$. Let $\gamma: [0,1] \to W^s(v)$ be
the lift of $\beta$ in the strong stable manifold and $\beta_t = \pi
\Phi^t \gamma$, where $\Phi^t$ is the geodesic flow on $SX$. We
conclude from \cite[Corollary 2.6]{KnPe2} that
$$ {\rm length}(\beta_t) \le a_2 r e^{-\frac{\rho}{2}t} $$
for all $t \ge 0$. Hence ${\rm length}(\beta_t) \leq 1$ for all
$$ t \ge t_0 := \frac{2 \log(a_2 r)}{\rho}. $$

\begin{figure}[h]
  \begin{center}
    \psfrag{p=pi(v)}{$p=\pi(v)$} 
    \psfrag{v}{$v$}
    \psfrag{Hv=piWsv}{${\mathcal H}_v=\pi W_v^s$}
    \psfrag{Hptv=piWsptv}{${\mathcal H}_{\Phi^t(v)}=\pi W_{\pi^t(v)}^s$}
    \psfrag{phit(v)}{$\Phi^t(v)$}
    \includegraphics[width=8cm]{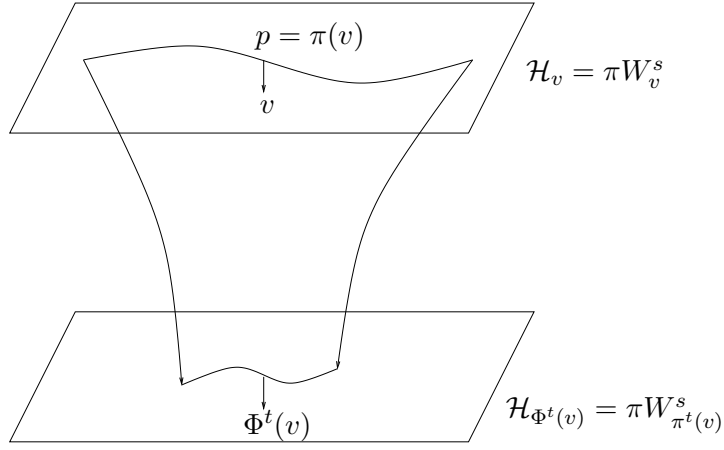}
  \end{center}
  \caption{Contraction of the geodesic flow on stable horospheres}
  \label{fig_horoshrink}
\end{figure}

Since the curvature of $X$ and the second fundamental form of horospheres
are bounded, the Gauss equation implies that the sectional curvatures of
horospheres are bounded, as well. Therefore, by the volume comparison
theorem, any ball of radius $1$ in any horosphere has an
intrinsic volume bounded by some constant $A>0$:
$$
\vol_{\mathcal{H}} (B_1(q)) \leq A \quad \forall\; \mathcal{H}\;
\text{horospheres}\; \forall\; q \in \mathcal{H}.
$$
This implies that
\begin{eqnarray*}
\vol_{\mathcal{H}_{v}} (B_r(p)) & \leq & 
\vol_{\mathcal{H}_{v}} (\Phi^{-t_0} (B_1 (\pi \circ \Phi^{t_0}(v))))\\
& \leq & e^{ht_0} \vol_{\mathcal{H}_{\Phi^{t_0}(v)}} (B_1 (\pi \circ \Phi_{t_0}(v))) 
\leq A e^{ht_0} = A' r^{\frac{2h}{\rho}},
\end{eqnarray*}
with $A' = A a_2^{\frac{2h}{\rho}}$. This proves that all horospheres
have polynomial volume growth in $(X,g)$. $\hfill \square$

%\section{Mean value property of harmonic functions at infinity}
%\label{chp:meanvalueinf}

\section{Horospherical means and bounded eigenfunctions}
\label{chp:meanvalueinf}

In this section, we are mainly concerned with horospherical means of
bounded eigenfunctions on rank one asymptotically harmonic manifolds
$X$ satisfying \eqref{eq:curvboundcond}. Before we consider the
special class of eigenfunctions, we first state a general result for
{\em all} continuous functions on the geometric compactification
$\overline{X}$. The underlying space is also more general than just
rank one asymptotically harmonic manifolds.

\begin{theorem} \label{thm:meanvalprop0} Let $(X,g)$ be a complete,
  simply connected Riemannian manifold without conjugate points of
  dimension $n$. Assume that the geometric compactification
  $\overline{X} = X \cup X(\infty)$ carries a topology such that the
  maps $\bar \varphi_p: \overline{B_1(p)} \to \overline{X}$ are
  homeomorphisms for all $p \in X$ (see Section \ref{sec:geomcomp} for
  details). Moreover, we assume that the following holds for every
  horosphere $\mathcal{H} \subset X$:
  \begin{itemize}
  \item[(a)] We have $\vol_{n-1}({\mathcal H}) = \infty$.
  \item[(b)] For every ball $B_r(p)$ of radius $r > 0$ around $p \in X$, we have
  $$ \vol_{n-1}({\mathcal H} \cap B_r(p)) < \infty. $$
\item[(c)] The closure of $\mathcal{H}$ in the geometric
  compactification $\overline{X}$ satisfies
  $$ \overline{\mathcal H} = \mathcal{H} \cup \{ \xi \}, $$
  where $\xi \in X(\infty)$ is the center of $\mathcal H$.
  \end{itemize}
  Then we have for every horosphere $\mathcal{H} \subset X$ centered
  at $\xi \in X(\infty)$, every compact exhaustion $\{ K_j \}$, and every
  continuous function $f: \overline{X} \to {\mathbb R}$: 
  \begin{equation} \label{eq:horosphmean}
  \lim\limits_{j \to \infty} \frac{\int_{K_j} f(x)
    dx}{\vol_{n-1}(K_j)} = f(\xi).
  \end{equation}
\end{theorem}

\begin{proof} Let $\mathcal{H}$ be centered at $\xi \in X(\infty)$ and
$p_0 \in X$. We show first indirectly that for every open
neighbourhood $U \subset \overline{X}$ of $\xi$ there exists $R > 0$
such that
\begin{equation} \label{eq:horoinc}
  \mathcal{H} \subset B_R(p_0) \cup U. 
\end{equation}
Assume that there exists $x_n \in \mathcal{H}$ with $x_n \not\in U$
and $d(p_0,x_n) \to \infty$. Then, after choosing a subsequence if
necessary, we have $x_n \to \xi' \in X(\infty)$ with $\xi' \neq \xi$. But
this is ruled out by (c). 

Let $\{ K_j \}$ be a compact exhaustion and $\epsilon > 0$ be given. Then there
exists an open neighbourhood $U \subset \overline{X}$ of $\xi$ such that
$$ | f(q) - f(\xi) | < \epsilon \quad \text{for all $q \in U$.} $$
Let $R > 0$ such that \eqref{eq:horoinc} is satisfied. Let $K_{j,0} =
K_j \cap B_r(p_0)$ and $K_{j,1} = K_j \backslash K_{j,0} \subset
U$. Then (a) and (b) yield $1/\vol(K_j) \int_{K_{j,0}} f \to 0$
and $\vol(K_{j,1})/\vol(K_j) \to 1$, which imply
$$ f(\xi)-\epsilon \le \liminf_{j \to \infty} \frac{\int_{K_j} f(x)
    dx}{\vol_{n-1}(K_j)} \le \limsup_{j \to \infty} \frac{\int_{K_j}
    f(x) dx}{\vol_{n-1}(K_j)} \le f(\xi)+\epsilon. $$
This shows \eqref{eq:horosphmean}, since $\epsilon > 0$ was arbitrary.
\end{proof}

The following proposition states that Theorem \ref{thm:meanvalprop0}
is applicable in our setting of rank one asymptotically harmonic
manifolds.

\begin{prop} \label{prop:horoprops} Let $(X,g)$ be a rank one
  asymptotically harmonic manifold of dimension $n$ satisfying
  \eqref{eq:curvboundcond}. Then every horosphere $\mathcal{H} \subset
  X$ satisfies properties (a), (b), and (c) in Theorem
  \ref{thm:meanvalprop0}.
\end{prop}

Before we present the proof of the proposition, we first introduce
some useful notation. Let ${\mathcal H} \subset X$ be a
horosphere. Then there exists $p_0 \in {\mathcal H}$ and $v \in
S_{p_0}X$ such that $\mathcal{H} = b_v^{-1}(0)$. Let $\mathcal{H}_t =
b_v^{-1}(t)$ and $\eta_t : X \to X$ be the flow associated to $\grad
b_v$. Then $\mathcal{H} = \mathcal{H}_0$, $\eta_t : \mathcal{H}_0 \to
\mathcal{H}_t$ and, for every $A \subset {\mathcal H}_0$ and $A(t) =
\eta_t(A) \subset {\mathcal{H}}_t$ we have (see
\cite[Prop. 3.1]{PeSa})
\begin{equation} \label{eq:AtA}
\vol_{n-1} (A(t)) = e^{ht} \vol_{n-1} (A).
\end{equation}
 
\begin{proof}
  (a) Assume there is a horosphere $\mathcal{H} \subset X$ with
  $\vol_{n-1}(\mathcal{H}) < \infty$. Using the above notation
  associated to $\mathcal{H}$, we see that the horoball
  $$ {\mathcal B} = \bigcup_{t \le 0} \mathcal{H}_t = b_v^{-1}((-\infty,0]) $$
  must also be of finite volume, since
  $$ \vol_n({\mathcal B}) = \int_{-\infty}^0 e^{ht}\, dt 
  \vol_{n-1}(\mathcal{H}) = \frac{1}{h} \vol_{n-1}(\mathcal{H}). $$ 
  But ${\mathcal B}$ contains the balls $B_r(c_v(r)) \subset X$ with
  arbitrarily large radii $r > 0$, whose volumes become arbitrarily large
  because of Proposition \ref{prop:Avest}. This is a contradiction.

  (b) Let $\mathcal{H}$ be a horosphere and $B_r(p) \subset X$ a
  ball. Let $A = {\mathcal H} \cap B_r(p)$ and assume that
  $\vol_{n-1}(A) = \infty$. Let $A_1 = \bigcup_{0 \le t \le 1}
  A(t)$. Then we also have $\vol_n(A_1) = \infty$, by
  \eqref{eq:AtA}. But $A_1 \subset B_{r+1}(p)$, and $B_{r+1}(p)$ has
  finite volume. This is, again, a contradiction.

  (c) Since ${\mathcal H} = b_v^{-1}(0)$ is closed in $X$, we only
  need to show that $\mathcal H$ has no other accumuluation points in
  $X(\infty)$ other than $\xi$. We proceed indirectly. Assume there
  exist $x_n \in \mathcal H$ with $d(p,x_n) \to \infty$ and $\lim x_n
  = \xi' \in X(\infty)$ and $\xi' \neq \xi$. Then we can find $\delta
  > 0$ such that $\xi' = c_w(\infty)$ for some $w \in S_{p_0}X$ with
  $\angle(w,v) > \delta$. Using the remark after Proposition
  \ref{prop:horoincone}, we know that there exists $s < 0$ such that
  $\mathcal{H}_s = \eta_s(\mathcal{H}) \subset C(v,\delta)$. Let
  $x_n(s) = \eta_s(x_n) \in \mathcal{H}_s$. Since $d(x_n,x_n(s)) = s$,
  we still have $x_n(s) \to \xi'$ and $x_n(s) \in \mathcal{H}_s
  \subset C(v,\delta)$ and, therefore, $\angle(w,v) \le \delta$,
  which is a contradiction.
\end{proof}

Next we prove the main result of this section for bounded
eigenfunctions (see Theorem \ref{thm:horomeanbdeigfunc} in the
Introduction). The proof is similar to the proof of Theorem 1 in \cite{KP}.

\begin{thm3} Let $(X,g)$ be a rank one asymptotically harmonic
  manifold of dimension $n$ satisfying \eqref{eq:curvboundcond} and $h
  > 0$ be the mean curvature of all horospheres. Let $\lambda \neq 0$
  be a real number and $f \in C^\infty(X)$ be a bounded
  function satisfying $\Delta f + \lambda f = 0$ and $\mathcal{H}
  \subset X$ be a horosphere with {\em isoperimetric exhaustion} $\{
  K_j \}$. Then we have
  \begin{equation} \label{eq:horosphmean2} \lim\limits_{j \to \infty}
    \frac{\int_{K_j} f(x) dx}{\vol_{n-1}(K_j)} = 0.
  \end{equation}
\end{thm3}

\begin{remark} Since horospheres have polynomial volume growth, the
  intrinsic balls of suitably chosen increasing radii $r_j$ satisfy
  $$
  \frac{\vol_{n-2}(\partial
    B_{\mathcal{H}}(r_j))}{\vol_{n-1}(B_{\mathcal{H}}(r_j))} \to 0.
  $$
  A suitable choice of sets $K_j$ are regularized spheres, as
  explained in \cite[p. 665]{KP}. But there might be many more
  increasing sets satisfying this asymptotic isoperimetric property.
\end{remark}

\begin{proof}
  We give an indirect proof. Assume that \eqref{eq:horosphmean2} is
  not satisfied. Then we can assume -- by replacing $\{ K_j \}$ by a
  subsequence, if needed -- that there exists $c \neq 0$ such that
  $$ \lim\limits_{j \to \infty}
    \frac{\int_{K_j} f(x) dx}{\vol_{n-1}(K_j)} = c.
  $$

  Let $\eta_t: {\mathcal H}_0 \to {\mathcal H}_t$ be again the flow
  defined above after Proposition \ref{prop:horoprops}. Let $K_j(t) =
  \eta_t (K_j) \subset \mathcal{H}_t$. Recall that we have
  $$
  \vol_{n-1} (K_j(t)) = e^{ht} \vol_{n-1} (K_j).
  $$
  Since $X$ has a lower sectional curvature bound, there exists $C >
  0$ such that
  $$
  \vol_{n-2}(\partial K_j(t)) \leq e^{C|t|} \vol_{n-2}(\partial K_j).
  $$
  This implies that, {\em on every compact interval} $I \subset [0,
  \infty)$, we have
  $$
  \left\Vert \frac{\vol_{n-2}(\partial
    K_j(\cdot))}{\vol_{n-1}(K_j(\cdot))} \right\Vert_{\infty,I} \to 0, \quad
  \text{as}\: j \to \infty.
  $$
  Define
  $$
  g_j(t) = \frac{\int_{K_j(t)} f(x)dx}{\vol_{n-1}(K_j(t))} \quad
  \forall\; t \in {\R}.
  $$
  Since $\Vert g_j \Vert_\infty \leq \Vert f \Vert_\infty$, using
  diagonal arguments, we find a subsequence $g_{j_k}$ such that
  $g_{j_k}(t) \to g(t)$, for all rational $t \in {\mathbb Q}$. Since
  $|\nabla f|$ is uniformly bounded by Yau's gradient estimate
  \cite[Theorem 3]{Yau}, $f$ is uniformly continuous and therefore,
  the sequence $g_{j_k}$ is equicontinuous. This implies that we have
  $g_{j_k} \to g$ pointwise to a continuous limit and $g(0) = c \neq
  0$.

  Next we show that $g$ satisfies
  \begin{equation} \label{eq:diffeq}
    g'' + hg' + \lambda g= 0, 
  \end{equation}
  in the distributional sense. Let $\psi \in C_0^\infty ({\R})$ be a
  test function. Then we have
  $$ \int\limits_{- \infty}^\infty g_j(t) (\psi''(t) - h
  \psi'(t) + \lambda h)dt = \int\limits_{- \infty}^\infty \frac{\int_{K_j(t)}
    f(x)dx}{\vol_{n-1}(K_j(t))} (\psi''(t) - h
  \psi'(t) + \lambda \psi)dt.
  $$
  Let $\tilde{f} : \mathcal{H} \times (- \infty,\infty) \to {\R}$ be
  defined as $\tilde{f}(x,t) : = f(\eta_t(x))$. The tranformation
  formula yields:
  $$
  \int\limits_{K_j(t)} f(x)dx = \int\limits_{K_j} f \circ \eta_t(x)
  \overbrace{{\rm Jac}\; \eta_t(x)}^{e^{ht}}dx = e^{ht}
  \int\limits_{K_j} \tilde{f} (x,t)dx.
  $$
  Therefore, we have $g_j(t) = 1/\vol_{n-1}(K_j) \int\limits_{K_j}
  f(\eta_t x)dx$, and
  \begin{eqnarray*}
    g_j''(t) + h g_j'(t) + \lambda g_j & = & \frac{1}{\vol_{n-1}(K_j)} 
    \int\limits_{K_j} \frac{d^2}{dt^2} f(\eta_t x) + h \frac{d}{dt} 
    f(\eta_t x)+ \lambda f(\eta_t x)dx\\
    & = & \frac{1}{\vol_{n-1}(K_j(t))} \int\limits_{K_j(t)} 
    \underbrace{\Delta_x f(x) + \lambda f(x)}_{= 0} - 
    \Delta_{\mathcal{H}_t} f(x)dx\\
    & = & - \frac{1}{\vol_{n-1}(K_j(t))} \int\limits_{K_j(t)} 
    \Delta_{\mathcal{H}_t} f(x)dx\\
    & = & \frac{1}{\vol_{n-1}(K_j(t))} \int\limits_{\partial K_j(t)} 
    \langle \grad_{\mathcal{H}_t} f(x), \nu_x \rangle dx,
  \end{eqnarray*}
  where $\nu_x$ denotes the outward unit vector of $\partial K_j(t)
  \subset {\mathcal{H}_t}$. Since ${\rm supp}\; \psi \subset \R$ is
  compact, we have
  \begin{multline*}
    \int\limits_{- \infty}^\infty g_j(t)(\psi''(t) - h \psi'(t) +
    \lambda \psi(t))dt = \int\limits_{- \infty}^\infty (g_j''(t) + h
    g_j'(t) + \lambda g_j(t))
    \psi(t)dt \\
    = \int\limits_{- \infty}^\infty \frac{1}{\vol_{n-1}(K_j(t))}
    \int\limits_{\partial K_j(t)} \langle \grad_{\mathcal{H}(t)} f(x),
    \nu_x \rangle dx\, \psi(t)dt.
  \end{multline*}
  Taking absolute value and using, again, Yau's gradient estimate
  \cite[Theorem 3]{Yau}, we obtain
  \begin{multline*}
    \left| \int\limits_{- \infty}^\infty g_j(t)(\psi''(t) - h
      \psi'(t) + \lambda \psi(t))dt \right| \\
    \leq \int\limits_{\supp \psi} \frac{\vol_{n-2} (\partial
      K_j(t))}{\vol_{n-1}(K_j(t))} \;  \Vert \grad_X f \Vert_\infty\;
    \Vert \psi \Vert_\infty\; dt \to 0,
  \end{multline*}
  as $j \to \infty$. By Lebesgue's dominated convergence, and since
  $\Vert g \Vert_\infty, \Vert g_j \Vert_\infty \leq \Vert f
  \Vert_\infty$, we conclude that
  $$
  \int\limits_{- \infty}^\infty g(t)(\psi''(t) - h \psi'(t) + \lambda
  \psi(t))dt = 0,
  $$
  i.e., the continuous function $g$ satisfies \eqref{eq:diffeq} in the
  distributional sense. Therefore, $g$ is smooth and satisfies
  $g''+hg'+\lambda g=0$ in the classical sense. This implies that $g$
  is of the general form
  \begin{equation} \label{eq:g}
  g(t) = 
  c_1 e^{\left( - \frac{h}{2} + \sqrt{\left( \frac{h}{2}  \right)^2 - \lambda} \right)t} + 
  c_2 e^{\left( - \frac{h}{2} - \sqrt{\left( \frac{h}{2}  \right)^2 - \lambda} \right)t}
  \end{equation}
  if $\lambda \neq (h/2)^2$ and
  $$ g(t) = c_1 e^{-\frac{h}{2}t} + c_2 t e^{-\frac{h}{2}t} $$
  if $\lambda = (h/2)^2$. It is straightforward to check for $\lambda
  \neq 0$ that every choice of $(c_1,c_2) \neq (0,0)$ leads to an unbounded
  function $g(t)$. But $g$ must be bounded because of $\Vert g
  \Vert_\infty \le \Vert f \Vert_\infty$. Therefore we conclude that
  $(c_1,c_2) = (0,0)$ in contradiction to $g(0) = c \neq 0$, finishing
  the indirect proof.
\end{proof}

\medskip

\begin{exbdeig} 
  (a) Let $X$ be a rank one symmetric space of non-compact type and $M
  = X / \Gamma$ be a compact quotient. Then every non-constant
  $\Delta_M$-eigenfunction $f \in C^\infty(M)$ gives rise to a bounded
  lift $\widetilde f \in C^\infty(X)$ which is also a
  $\Delta_X$-eigenfunction to the same eigenvalue.  Since $\widetilde
  f$ is non-constant and $\Gamma$-periodic, it does not admit a
  continuous extension to the compacitification $\overline{X}$.

  (b) Let $X^{(p,q)}$ be a Damek-Ricci space with $p,q$ defined as in
  \cite{Rou}.  Then $X^{(p,q)}$ is an asymptotically harmonic manifold
  with $h = p/2 + q$ and there exist radial eigenfunctions
  $\varphi_\mu \in C^\infty(X^{(p,q)})$ satisfying
  $$ \Delta \varphi_\mu  + \left( \mu^2 + 
  \left( \frac{h}{2} \right)^2 \right) \varphi_\mu = 0 \quad 
  \text{and $\varphi_\mu(e) = 1$}, $$
  where $\mu \in {\mathbb C}$ and $e \in X^{(p,q)}$ denotes the
  neutral element in the Damek-Ricci space considered as a solvable
  group. If $ 0 < i \mu < h/2$, we have (see
  \cite[p. 78]{Rou})
  $$ \varphi_\mu(r) \sim c(\mu) e^{(i\mu-h/2)r} \qquad 
  \text{as $r \to \infty$} $$ 
  with suitable constants $c(\mu) \in {\mathbb R} \backslash
  \{0\}$. This means that $\varphi_\mu$ is a bounded eigenfunction
  with {\em trivial} continuous extension to the compactification
  $\overline{X^{(p,q)}}$.
\end{exbdeig}

Now we are in a position to prove our final result (see Theorem
\ref{thm:dirichprob-eigen} in the Introduction) which states that the above examples are the only
two possible cases with regards to continuous extensions of bounded eigenfunctions $f$:
either $f$ cannot be extended to $\overline{X}$ or the extension is trivial.

\begin{thm4} Let $(X,g)$ be a rank one asymptotically harmonic
  manifold satisfying \eqref{eq:curvboundcond}. Let $\lambda \in
  {\mathbb R} \backslash \{0\}$ and $f \in C^\infty(X)$ be an
  eigenfunction $\Delta f + \lambda f = 0$. If $f$ has a continuous
  extension $F \in C(\overline{X})$ then we have necessarily $F
  \vert_{\partial X} \equiv 0$.
\end{thm4}

\begin{proof}
  Assume that $\lambda \neq 0$ and that an eigenfunction $\Delta f +
  \lambda f = 0$ has a continuous extension $F$ on the
  compactification $\overline{X}$. Then we know from Theorem
  \ref{thm:meanvalprop0} that all horospherical means of $f$ over
  horospheres centered at $\xi \in X(\infty)$ agree with $F(\xi)$.  On
  the other hand, we conclude from Theorem \ref{thm:horomeanbdeigfunc}
  that all horospherical means with isoperimetric exhaustions have to
  vanish. Moreover, every horosphere in $X$ has polynomial volume
  growth and, therefore, admits isoperimetric exhaustions. This implies
  that $F \vert_{\partial X} \equiv 0$.
\end{proof}

\end{document}